\documentclass[11pt]{article}
\usepackage[T1]{fontenc}
\usepackage{amsmath}
\usepackage{amsfonts}
\usepackage{amssymb}
\usepackage{amsthm}
\usepackage{graphicx}
\graphicspath{{img/}}
\usepackage{booktabs}
\usepackage{multicol}
\usepackage{bbm}
\usepackage[ruled,vlined]{algorithm2e}
\usepackage{adjustbox}
\usepackage{multirow}
\usepackage{subfig}
\usepackage{pgfplots}
\usepackage[breaklinks=true]{hyperref}
\usepackage{breakcites}

\newtheorem{theorem}{Theorem}
\newtheorem{remark}{Remark}
\newtheorem{corollary}{Corollary}
\newtheorem{definition}{Definition}
\newtheorem{proposition}{Proposition}
\newtheorem{example}{Example}

\pgfplotsset{compat=1.17}
\usetikzlibrary{calc}

\usepackage[margin=1.00in]{geometry}

\newcommand{\laplace}[2]{\mathcal{L}_{#1}(#2)}
\newcommand{\diff}{\, \mathrm{d}}

\title{Risk aggregation with FGM copulas}

\author{Christopher Blier-Wong\thanks{Corresponding author, \href{mailto:chblw@ulaval.ca}{chblw@ulaval.ca}},
Hélène Cossette and Etienne Marceau\\ École d'actuariat, Université Laval, Québec, Canada}

\date{\today}

\begin{document}

\maketitle

\begin{abstract}
We offer a new perspective on risk aggregation with FGM copulas. Along the way, we discover new results and revisit existing ones, providing simpler formulas than one can find in the existing literature. This paper builds on two novel representations of FGM copulas based on symmetric multivariate Bernoulli distributions and order statistics. First, we detail families of multivariate distributions with closed-form solutions for the cumulative distribution function or moments of the aggregate random variables. We order aggregate random variables under the convex order and provide methods to compute the cumulative distribution function of aggregate rvs when the marginals are discrete. Finally, we discuss risk-sharing and capital allocation, providing numerical examples for each. 
\end{abstract}


\section{Introduction}

Insurance companies deals with a large number of heterogeneous and possibly dependent losses. For enterprise risk management purposes, it is important to understand the risks in one's portfolio at the individual level, but also at the company-wide level. For this reason, one is interested in the aggregate risk of the portfolio. 

In this paper, we aim to provide a comprehensive treatment of risk aggregation of positive random variables (rvs) when the dependence structure is a Farlie-Gumbel-Morgenstern (FGM) copula. The family of FGM copulas has a long history in copula theory (see, for instance, \cite{johnson1975generalized}, \cite{cambanis1977some}, \cite[Chapter 5]{kotz2001correlation}, \cite[Section 44.10]{kotz2004continuous},\cite{nelsen2007introduction}, \cite{durante2015principles}). The family of FGM copulas is a popular copula since its simple shape enables analytic results, see, for instance, \cite{genest2007everything}. One finds applications of the FGM family of copulas in actuarial science (for instance, \cite{cossette2008compound, barges2009tvar, barges2011moments, cossette2012tvarbased, cossette2013multivariate, woo2013note, chadjiconstantinidis2014renewal}). A FGM copula admits weak dependence, both positive and negative. For instance, the range of bivariate Spearman's rho for FGM copulas is $[-1/3, 1/3]$. 

Within the context of a large portfolio of diversified insurance risks, one does not expect to observe high dependence across every risk. Indeed, important conditions for insurability includes having a large number of similar exposure units, and a limited exposure to catastrophically large losses. An insurance company would actively avoid insuring two risks that exhibit significant positive dependence. For this reason, many insurance companies limit their exposures in regions where a single event could cause multiple claims. However, insurers do not refuse a risk simply because they have another risk that is positively correlated with a potential customer; weak positive dependence may be acceptable within the underwriting guidelines of an insurance company. A FGM copula therefore seems appropriate for a large portfolio of insurance risks because one expects underwriters to limit positive dependence, and a FGM copula lets one select flexible dependence structures between risks within the portfolio, under parameter constraints that the underlying FGM copula exists.

Risk aggregation under FGM dependence has already been studied in the actuarial science literature, (see, for instance, \cite{barges2009tvar, cossette2013multivariate, cossette2015two, navarro2020copula}), but we consider the problem from a new perspective. In the past, FGM copulas did not have a genuine probabilistic interpretation (see, for instance, \cite{durante2012method}). This paper builds on two alternate representations of the FGM family of copulas that provide probabilistic interpretations. The first representation provides a method to construct FGM copulas, interpret the copula parameters, and enables the stochastic comparison of different FGM copulas. The second representation, for a given set of copula parameters, leads to new results on risk aggregation, and rediscovers some that were cited in the literature above. 
\begin{itemize}
	\item The first representation is based on a one-to-one correspondence between the class of FGM copulas and symmetric multivariate Bernoulli random vectors, explored in \cite{blier-wong2022stochastic}. By constructing a $d$-variate FGM copula from a $d$-variate symmetric multivariate Bernoulli random vector $\boldsymbol{I}$, we will see in Section \ref{sec:prelim} that the dependence structure of $\boldsymbol{I}$ governs the dependence structure of the FGM copula. One significant advantage of this representation is that the dependence structure of Bernoulli rvs are easier to interpret than a set of $2^d - d - 1$ central mixed moments between $k$-tuples, for $2 \leq k \leq d$, which is what one has with the natural formulation of the FGM copula. Another advantage of this representation is that it enables one to answer such questions as (i) what is the most positive and negative dependence structure attainable under FGM dependence; (ii) what is the effect of increasing a certain dependence parameter on the resulting aggregate distribution; (iii) how are two aggregate distributions with different FGM copulas ordered under the convex order. It turns out that trying to answer these questions using the natural representation of the FGM copula is tedious, but becomes simple when using the stochastic representation. 
	\item The second representation is based on order statistics. In \cite{baker2008orderstatisticsbased}, the author constructs multivariate distributions based on order statistics, and finds that the simplest case 
	consisting of mixing the order statistics from two independent and identically distributed (iid) rvs corresponds to a FGM distribution. See also Section 8.3 of \cite{bladt2017matrixexponential} for construction of multivariate models based on order statistics. It follows that if the order statistics of the marginal distributions have convenient forms, the aggregate distribution of the risks under FGM dependence may also have convenient forms. Instead of approaching the problem or risk aggregation from a purely mathematical point of view, we approach it using a probabilistic argument that simplifies the formulas and provides a more straightforward interpretation of the resulting expressions. 
\end{itemize}

While FGM copulas only admit a moderate strength of dependence, we show that the dependence structure still has a significant impact on the distribution of the aggregate rv. Another advantage is that FGM copulas admit a wide variety of shapes, that is, a $d$-dimensional copula has $2^d - d - 1$ copula parameters, each parameter controls the moments between $k$-tuples of the random vector, for $k \in \{2, \dots, d\}$. Hence, we may study the effect of mild negative and positive dependence on the behaviour of the aggregate rv within the family of FGM copulas. Also, FGM copulas are the most simple case of Bernstein copulas, introduced in \cite{sancetta2004bernstein}. Bernstein copulas of dimension $d$ are interesting from a practical point of view since they are dense on the hypercube $[0, 1]^d$. It follows that one may use Bernstein copulas to approximate other types of copulas. The results from this paper, covering FGM copulas, consist of important groundwork to study risk aggregation under a dependence structure induced by Bernstein copulas. See, e.g., \cite{marri2021risk} for related research for risk aggregation with mixed Bernstein copulas. 

The remainder of this paper is structured as follows. In Section \ref{sec:prelim}, we provide the preliminary notions of copulas and order statistics required for the main results of the paper. Section \ref{sec:general} outlines the general method to identify the Laplace-Stieltjes transform or the $m$th moments $m \in \mathbb{N}_1$, for the aggregate rv. In Section \ref{sec:continuous}, we develop closed-form expressions for the cdf and $m$th moments for some continuous rvs. We then deal with stochastic orders in Section \ref{sec:stochastic-order}, identifying the lower and upper bounds of the aggregate rv under the convex order for the special case of exchangeable FGM copulas. Section \ref{sec:discrete} details a method to compute the probability mass function (pmf) of the aggregate rv when each marginal is a discrete rv. In Section \ref{sec:approx}, we propose a method to approximate the cumulative distribution function (cdf) of continuous rvs using discretization methods, and construct bounds for the risk measures of the aggregate rv by their discrete counterparts. Section \ref{sec:allocation} discusses TVaR-based risk allocation when the marginals are mixed Erlang rvs. In Section \ref{sec:conclusion}, we discuss the results and present some openings to further research.

\section{Preliminaries}\label{sec:prelim}

We begin by introducing general notation. Let $\boldsymbol{x}$ denote a vector $(x_1, \dots, x_d) \in \mathbb{R}^d$. All expressions such as $\boldsymbol{x} + \boldsymbol{y}$, $\boldsymbol{x} \times \boldsymbol{y}$ and $\boldsymbol{x} \leq \boldsymbol{y}$ represent component-wise operations. Let $\boldsymbol{X}$ represent a random vector on $\mathbb{R}_+^d$ with joint cdf $F_{\boldsymbol{X}}$ with $F_{\boldsymbol{X}}(\boldsymbol{x}) = \Pr(X_1 \leq x_1, \dots, X_d \leq x_d)$. Define also the Laplace-Stieltjes transform (LST) as $\mathcal{L}_{\boldsymbol{X}}(\boldsymbol{t}) = E\left[\exp\left\{-(t_1X_1  + \dots + t_dX_d)\right\}\right]$, for $\boldsymbol{t} \in \mathbb{R}_+^d$. Also, for a univariate cdf $F_X$, we define the generalized inverse $F_X^{-1}(u) = \inf \left\{x \in \mathbb{R}, F_X(x) \geq u\right\}$. Let the cdf of a symmetric Bernoulli distribution be denoted by $F_I(x) = 0.5 \times 1_{[0, \infty)}(x) + 0.5 \times 1_{[1, \infty)}(x), x \geq 0$, where $1_{A}(x) = 1$, if $x \in A$ and 0, otherwise. We denote $\mathcal{B}_d$ as the Fréchet class with univariate marginals $F_{I}$. Finally, we let $\mathbb{N}_0$ be the set of non-negative integers and $\mathbb{N}_1 = \mathbb{N}_0\setminus \{0\}$ be the set of strictly positive integers.

\subsection{Order statistics}

This paper leverages the order statistic representation of the FGM copula, presented in \cite{baker2008orderstatisticsbased} and revisited in Section 8.3 of \cite{bladt2017matrixexponential} and more recently in \cite{blier-wong2022stochastic}. The current section provides preliminary results for order statistics. The interested reader can refer to the standard references on order statistics, for example \cite[Section 5.4]{casella2002statistical}, \cite{david2003order} and \cite{arnold2008first}, for more details. 

Let $(X_1, X_2)$ be a vector of two continuous iid rvs with marginal cdf $F_X$ and probability density function (pdf) $f_X$. Define the vector $(X_{[1]}, X_{[2]})$ as the vector of order statistics of $(X_1, X_2)$, that is, $X_{[1]} = \min (X_1, X_2)$ and $X_{[2]} = \max (X_1, X_2)$. The cdfs and pdfs of the order statistics are 
\begin{align}
	F_{X_{[1]}}(x) &= 1 - \overline{F}_X(x)^{2}; \quad f_{X_{[1]}}(x) = 2\overline{F}_X(x)f_X(x); \label{eq:cdf-pdf-min}\\
	F_{X_{[2]}}(x) &= F_X(x)^{2}; \quad f_{X_{[2]}}(x) = 2F_{X}(x)f_{X}(x),\label{eq:cdf-pdf-max}
\end{align}
where $x$ takes values in the same support as those of $F_X$ or $f_X$. 

The following example presents the well known order statistics of exponential rvs, first derived in \cite{renyi1953theory}.
\begin{example}\label{ex:exp-order-stats}
	Let $(X_{1}, X_{2})$ be two independent and exponentially distributed rvs with mean $1/\beta$. Let $Z_i, i \in \{1, 2\}$ be independent exponentially distributed rvs with mean 1. The associated order statistics $X_{[1]}$ and $X_{[2]}$ admit the representation
	\begin{align*}
		X_{[1]} = \min(X_{1}, X_{2}) \overset{\mathcal{D}}{=}& \frac{Z_1}{2\beta};\\
		X_{[2]} = \max(X_{1}, X_{j, 2}) \overset{\mathcal{D}}{=}& \frac{Z_1}{2\beta} + \frac{Z_2}{\beta},
	\end{align*}
	where $\overset{\mathcal{D}}{=}$ means equality in distribution. It follows that $X_{[1]} \sim Exp(2\beta)$ and that $X_{[2]}$ follows a generalized Erlang distribution with parameters $\beta$ and $2\beta$.
\end{example}

Another useful representation of order statistics, due to \cite{scheffe1945non}, is
\begin{equation}\label{eq:order-stat-prob-int-transform}
    (X_{[1]}, X_{[2]}) \overset{\mathcal{D}}{=} \left(F^{-1}_{X}(U_{[1]}), F^{-1}_{X}(U_{[2]})\right).
\end{equation}
One can also write the pdf of two order statistics as 

\begin{equation}\label{eq:pdf-order-stat}
	f_{X_{[j]}}(x) = 2F_{X}(x)^{j-1}\overline{F}_X(x)^{2-j}f_{X}(x) = (-1)^{j}2f_{X}(x)F_{X}(x) + (2-j)f_{X}(x),\\
\end{equation}
for $j \in \{1, 2\}$. From the second equality in \eqref{eq:pdf-order-stat}, we have 
\begin{equation}\label{eq:pdf-order-stat2}
	f_{X_{[j]}}(x) = (-1)^{j}f_{X_{[2]}}(x) + 2(2-j)f_{X}(x), \quad j \in \{1, 2\}.
\end{equation}
Define $\mu_{X_{[j]}}^{(m)}$ as the $m$th moment of the $j$th order statistic of $X$. From \eqref{eq:pdf-order-stat2}, we have
\begin{equation}\label{eq:moment-m}
	\mu_{X_{[j]}}^{(m)} = (-1)^{j}\mu_{X_{[2]}}^{(m)} + 2(2-j)E[X^{m}], \quad m \in \mathbb{N}_0.
\end{equation}
Replacing $j = 1$ in \eqref{eq:moment-m} we have that the relationship between moments of order statistics is
\begin{equation}\label{eq:order-statistic-moment-relationship}
	E[X^{m}] = \frac{1}{2} \left(\mu_{X_{[1]}}^{(m)} + \mu_{X_{[2]}}^{(m)}\right), \quad m \in \mathbb{N}_0.
\end{equation}

In this paper, we will construct dependent random vectors by defining their joint cdfs with copulas. Standard references on copula theory include, for example, \cite{kotz2001correlation}, \cite{trivedi2006copula}, \cite{nelsen2007introduction}, \cite{mai2014financial}, \cite{joe2014dependence}, or \cite{durante2015principles}. Copulas are multivariate cdfs whose marginals are uniformly distributed on the interval $[0, 1]$, and the copula studied in this paper is constructed by pairs of order statistics. It is therefore useful to recall that for the special case where $(U_1, U_2)$ is a pair of iid uniform rvs, then the order statistics $U_{[1]}$ and $U_{[2]}$ satisfy $U_{[j]} \sim Beta(j, 3 - j)$, for $j \in \{1, 2\}$.

\subsection{FGM copulas}

In this paper, we focus on FGM copulas, whose expression is given by
\begin{equation} \label{eq:fgm-copula}
	C\left( \boldsymbol{u}\right) =\prod_{k=1}^d u_k \left( 1+\sum_{k=2}^{d}\sum_{1\leq j_{1}<\cdots <j_{k}\leq d}\theta_{j_{1}\ldots j_{k}}\overline{u}_{j_{1}}\overline{u}_{j_{2}}\ldots \overline{u}_{j_{k}}\right), \quad \boldsymbol{u}\in [0,1]^{d},
\end{equation}
where $\overline{u}_{j}=1-{u}_{j}$, for $j \in \{1,\dots,d\}$. We note $\mathcal{C}_d^{FGM}$ as the class of $d$-variate FGM copulas. The constraints on the parameters for the copula in \eqref{eq:fgm-copula}, as derived by \cite{cambanis1977some}, are
\begin{equation}\label{eq:constraints-general}
	\mathcal{T}_d = \left\{(\theta_{12}, \dots, \theta_{1\dots d}) \in \mathbb{R}^{2^d - d - 1} : 1+\sum_{k=2}^{d}\sum_{1\leq j_{1}<\dots <j_{k}\leq d}\theta_{j_{1}\dots j_{k}}\varepsilon _{j_{1}}\varepsilon _{j_{2}}\dots
	\varepsilon_{j_{k}}\geq 0\right\},  
\end{equation}
for $\{\varepsilon_{j_{1}},\varepsilon_{j_{2}},\dots, \varepsilon_{j_{k}}\} \in \{-1,1\}^d$ and $k\in \{2, \dots, d\}$. When $d=2$, \eqref{eq:fgm-copula} becomes the well-known expression of the bivariate FGM copula with one parameter and given by
\begin{equation} \label{eq:biv-fgm-copula} 
	C(u_{1},u_{2}) = u_1 u_2 + \theta_{12} u_1 u_2 \overline{u}_1\overline{u}_2, \quad (u_1,u_2) \in [0,1]^2,
\end{equation}
with $\theta_{12} \in \mathcal{T}_2 = \left[-1,1\right]$. The association measures such as Kendall's tau and Spearman's rho for the bivariate FGM copula are respectively given by $\tau =2\theta_{12}/9$ and $\rho =\theta_{12}/3$. We use the notation $C \in \mathcal{C}_d^{FGM}$ to denote that $C$ is a FGM copula. 

The following Theorem combines the stochastic representation based on multivariate symmetric Bernoulli random vectors of FGM copulas proposed in \cite{blier-wong2022stochastic} along with the stochastic representation based on order statistics that is discussed in \cite{baker2008orderstatisticsbased} and Section 8.3.2 of \cite{bladt2017matrixexponential}.
\begin{theorem}\label{thm:stochastic-representation}
	The copula in \eqref{eq:fgm-copula} has the equivalent representation 
	\begin{equation}\label{eq:order-statistic-formulation}
		C(\boldsymbol{u}) = E_{\boldsymbol{I}}\left[\prod_{k = 1}^d F_{U_{[I_k + 1]}}(u_k)\right], \quad \boldsymbol{u} \in [0, 1]^d,
	\end{equation}
	for $\boldsymbol{u} \in [0, 1]^d$, where $f_{\boldsymbol{I}}$ is the pmf of $\boldsymbol{I}$, a symmetric multivariate Bernoulli random vector, that is, $F_{\boldsymbol{I}} \in \mathcal{B}_d$.
	The dependence parameters are proportional to central mixed moments as follows:
	\begin{equation}\label{eq:theta_from_I}
		\theta_{j_{1}\ldots j_{k}} = (-2)^k E_{\boldsymbol{I}}\left\{\prod_{n = 1}^k\left(I_{j_n} - \frac{1}{2}\right)\right\},
	\end{equation}
	for $k \in \{2, \dots, d\}$ and $1 \leq j_{1} < \dots < j_{k} \leq d$.
\end{theorem}
\begin{proof}
	The relation in \eqref{eq:order-statistic-formulation} follows from Corollary 3.3 of \cite{blier-wong2022stochastic} and the definitions of cdfs for order statistics in \eqref{eq:cdf-pdf-min} and \eqref{eq:cdf-pdf-max}. See also Remark 3.5 of \cite{blier-wong2022stochastic}.
\end{proof}

\begin{theorem}\label{thm:stochastic-sum}
	Let $\boldsymbol{U}_{[j]} = (U_{1, [j]}, \dots, U_{d, [j]})$ be a $d$-variate vector of iid rvs satisfying $U_{k, [j]} \sim Beta(j, 3 - j)$ for $k \in \{1, \dots, d\}$ and $j \in \{1, 2\}$. Define the random vector 
	\begin{equation}\label{eq:U-based-on-I}
		\boldsymbol{U} = (\boldsymbol{1}-\boldsymbol{I}) \boldsymbol{U}_{[1]} + \boldsymbol{I} \boldsymbol{U}_{[2]},
	\end{equation}
	where $\boldsymbol{1}$ is a $d$-variate vector of ones. Then, we have that $F_{\boldsymbol{U}}(\boldsymbol{u}) \in \mathcal{C}_d^{FGM}$. More generally, fix some marginal cdfs $F_{X_1}, \dots, F_{X_d}$ and let $\boldsymbol{X}_{[j]} = (X_{1, [j]}, \dots, X_{d, [j]})$ be vectors of independent rvs with respective marginal cdf $F_{X_{k, [j]}}$, as defined in \eqref{eq:cdf-pdf-min} and \eqref{eq:cdf-pdf-max}, for $k \in \{1, \dots, d\}$ and $j\in\{1, 2\}$. Define the random vector
	\begin{equation}\label{eq:X-based-on-I}
		\boldsymbol{X} = (\boldsymbol{1} - \boldsymbol{I}) \boldsymbol{X}_{[1]} + \boldsymbol{I}\boldsymbol{X}_{[2]}. 
	\end{equation}
	Then, we have $F_{\boldsymbol{X}}(\boldsymbol{x}) = C(F_{X_1}(x_1), \dots, F_{X_d}(x_d)),$ where $C \in \mathcal{C}_{d}^{FGM}$.
\end{theorem}
\begin{proof}
	We prove the statement about the random vector $\boldsymbol{X}$. We have
	\begin{align*}
		\Pr(\boldsymbol{X} \leq \boldsymbol{x})	&= E_{\boldsymbol{I}}\left[\Pr\left((1-I_1)X_{1, [1]} + I_1X_{1, [2]} \leq x_1, \dots, (1-I_d)X_{d, [1]} + I_dX_{d, [2]} \leq x_d\right)\right],
	\end{align*}
	which becomes
	\begin{align*}
	    \Pr(\boldsymbol{X} \leq \boldsymbol{x}) 
		&= E_{\boldsymbol{I}}\left[\prod_{k = 1}^{d}\Pr\left((1-I_k)X_{k, [1]} + I_kX_{k, [2]} \leq x_k\right)\right]\\
		&= E_{\boldsymbol{I}}\left[\prod_{k = 1}^{d}\Pr\left(X_{k, [I_k + 1]} \leq x_k\right)\right]= E_{\boldsymbol{I}}\left[\prod_{k = 1}^{d}F_{X_{k, [I_k + 1]}}(x_k)\right]\\
		&= E_{\boldsymbol{I}}\left[\prod_{k = 1}^{d}F_{U_{[I_k + 1]}}(F_{X_k}(x_k))\right].
	\end{align*}
	The proof for the random vector $\boldsymbol{U}$ holds by replacing $F_{X_k}(x) = x$ for $k = 1, \dots, d$.
\end{proof}

The representation of Theorems \ref{thm:stochastic-representation} and \ref{thm:stochastic-sum} are more convenient to develop the results of the current paper and will help us understand the dependence structure behind $C \in \mathcal{C}_d^{FGM}$. Theorem \ref{thm:stochastic-representation} states that, conditional on $\boldsymbol{I}$, the copula $C$ is the product of independent cdfs of $U_{[1]}$ or $U_{[2]}$. Theorem \ref{thm:stochastic-sum} constructs random vectors $\boldsymbol{U}$ and $\boldsymbol{X}$ which have the same joint cdfs as the ones we are investigating in this paper. The authors of \cite{blier-wong2022stochastic} call \eqref{eq:fgm-copula} the natural representation of the FGM copula since the parameters in \eqref{eq:theta_from_I} are central mixed moments. They also refer to \eqref{eq:order-statistic-formulation} as the stochastic representation of the FGM copula since it relies on the stochastic nature based on $\boldsymbol{I}$ and order statistics. 

\section{Risk aggregation with FGM copulas: the general method}\label{sec:general}

For this section, we consider a vector of rvs $\boldsymbol{X}  = (X_1, \dots, X_d)$ with cdf
\begin{equation}\label{eq:cdf-x}
	F_{\boldsymbol{X}}(\boldsymbol{x}) = C(F_{X_1}(x_1), \dots, F_{X_d}(x_d)), \quad \boldsymbol{x} \in \mathbb{R}_+^d,
\end{equation}
where $C \in \mathcal{C}_d^{FGM}$. Using representation of the FGM copula in \eqref{eq:order-stat-prob-int-transform} and \eqref{eq:order-statistic-formulation}, the joint cdf of $\boldsymbol{X}$ becomes
	\begin{equation}\label{eq:cdf-x-v2}
		F_{\boldsymbol{X}}(\boldsymbol{x}) = E_{\boldsymbol{I}}\left[\prod_{k = 1}^d F_{U_{[I_k + 1]}}(F_{X_k}(x_k))\right] = E_{\boldsymbol{I}}\left[\prod_{k = 1}^d F_{X_{k, [I_k+1]}}(x_k)\right], \quad \boldsymbol{x} \in \mathbb{R}^d.
	\end{equation}
	It follows that the joint LST of $\boldsymbol{X}$ is 
	\begin{equation}\label{eq:lst-x}
		\laplace{\boldsymbol{X}}{\boldsymbol{t}} = \int_{\mathbb{R}_+^d} e^{-\boldsymbol{tx}} \diff F_{\boldsymbol{X}}(\boldsymbol{x})
		= E_{\boldsymbol{I}}\left[\prod_{k = 1}^d \int_{\mathbb{R}_+} e^{-t_k x_k}\diff F_{X_{k, [I_k+1]}}(x_k)\right]
		= E_{\boldsymbol{I}}\left[\prod_{k = 1}^d \laplace{X_{k, [I_k + 1]}}{t_k}\right],
	\end{equation}
for $t \geq 0$. Let $S$ be the rv representing the aggregate risk of the vector $\boldsymbol{X}$, that is, $S = X_1 + \dots + X_d$. From Theorem \ref{thm:stochastic-sum}, we also have
\begin{align}
	S \overset{\mathcal{D}}{=}& \sum_{k = 1}^{d} \left\{(1-I_k) X_{k, [1]} + I_kX_{k, [2]}\right\}.\label{eq:S-representation-sum}\\
\end{align}
We are now in a position to state the following Theorem.
\begin{theorem}\label{prop:lst-s}
	The LST of $S$ is
	\begin{equation}\label{eq:lst-s}
		\laplace{S}{t} = E_{\boldsymbol{I}}\left[\prod_{k = 1}^d \laplace{X_{k, [I_k + 1]}}{t}\right], \quad t\geq 0.
	\end{equation}
\end{theorem}
\begin{proof}
	The result follows directly from the definition of $\mathcal{L}_S(t)$ and \eqref{eq:lst-x}.
\end{proof}

Theorem \ref{prop:lst-s} is the main tool to identify the distribution of the aggregate risk $S$. In some cases, we obtain exact results to compute the cdf $S$. In others, we are only able to obtain the moments of $S$.

\begin{theorem}
	For $m \in \mathbb{N}_1$, and assuming that $E[X_k^m]$ exists for $k = 1, \dots, d$, we have
	\begin{equation}\label{eq:m-moment-s-stochastic-max}
		E\left[S^m\right] = \sum_{j_1 + \cdots + j_d = m} \frac{m!}{j_1!\cdots j_d!} \left\{\prod_{k = 1}^{d}E[X_k^{j_k}]\right\} E_{\boldsymbol{I}}\left[\prod_{k = 1}^{d}\left\{1 + (-1)^{I_k}\left(1 - \frac{\mu_{X_{k, [2]}}^{(j_k)}}{E[X_k^{j_k}]}\right)\right\}\right].
	\end{equation}
\end{theorem}

\begin{proof}
	Applying the multinomial theorem, we have
	$$E\left[S^m\right] = E\left[\left(\sum_{k = 1}^{d} X_k\right)^m \right] =E\left[\sum_{j_1 + \dots + j_d = m}  \binom{m!}{j_1!\dots j_d!} X_1^{j_1} \dots X_d^{j_d}\right].$$
	We condition on $\boldsymbol{I}$ to obtain
	\begin{align*}
	    E[S^m] &= E_{\boldsymbol{I}}\left[ E\left[\left.\sum_{j_1 + \cdots + j_d = m} \frac{m!}{j_1!\cdots j_d!}  X_{1, [I_1 + 1]}^{j_1}\cdots X_{d, [I_d + 1]}^{j_d}\right\vert \boldsymbol{I}\right]\right]\\
		&= \sum_{j_1 + \cdots + j_d = m} \frac{m!}{j_1!\cdots j_d!}E_{\boldsymbol{I}}\left[\mu_{X_{1, [I_1 + 1]}}^{(j_1)}\cdots \mu_{X_{d, [I_d + 1]}}^{(j_d)}\right].
	\end{align*}
	Inserting the last equality into \eqref{eq:moment-m}, the $m$th moment of $S$ becomes
	$$E\left[S^m\right] = \sum_{j_1 + \cdots + j_d = m} \frac{m!}{j_1!\cdots j_d!}E_{\boldsymbol{I}}\left[\prod_{k = 1}^{d}\left\{(-1)^{1 + I_k}\mu_{X_{k, [2]}}^{(j_k)} + 2(1- I_k)E[X_k^{j_k}]\right\}\right].$$
	Factoring out the expected values of the original marginals yields the desired result.
\end{proof}
One can also use the relation in \eqref{eq:order-statistic-moment-relationship} to obtain 
\begin{equation}
	E\left[S^m\right]= \sum_{j_1 + \cdots + j_d = m} \frac{m!}{j_1!\cdots j_d!} \left\{\prod_{k = 1}^{d}E[X_k^{j_k}]\right\} E_{\boldsymbol{I}}\left[\prod_{k = 1}^{d}\left\{1 + (-1)^{I_k}\left(\frac{\mu_{X_{k, [1]}}^{(j_k)}}{E[X_k^{j_k}]} - 1\right)\right\}\right].\label{eq:m-moment-s-stochastic-min}
\end{equation}

One obtains exact results for the $m$th moment of $S$ if one has exact results for the $j$th moment of each marginal and either the minimum or maximum of two marginals, with $j \in \{1, \dots, m\}$. Alternatively, we can write the moments in terms of the natural representation of the FGM copula. 
\begin{corollary}\label{cor:moments}
	For $m \in \mathbb{N}_1$, we have
	\begin{equation}\label{eq:moments-natural-formulation}
		E\left[S^m\right] = \sum_{j_1 + \cdots + j_d = m} \frac{m!}{j_1!\cdots j_d!} \left\{\prod_{k = 1}^{d}E[X_k^{j_k}]\right\} A_l(j_1, \dots, j_k),
	\end{equation}
	for either $l \in \{1, 2\}$, with
	\begin{align*}
		A_1(j_1, \dots, j_k) = 1 + \sum_{k = 2}^{d} \sum_{1 \leq n_1 < \cdots < n_k \leq d}\theta_{n_1\dots n_k} \left(\frac{\mu_{X_{n_1, [1]}}^{(j_{n_1})}}{E[X_{n_1}^{j_{n_1}}]} - 1\right) \cdots \left(\frac{\mu_{X_{n_k, [1]}}^{(j_{n_k})}}{E[X_{n_k}^{j_{n_k}}]} - 1\right);\\
		A_2(j_2, \dots, j_k) = 1 + \sum_{k = 2}^{d} \sum_{1 \leq n_1 < \dots < n_k \leq d}\theta_{n_1\dots n_k} \left(1 - \frac{\mu_{X_{n_1, [2]}}^{(j_{n_1})}}{E[X_{n_1}^{j_{n_1}}]}\right) \cdots \left(1 - \frac{\mu_{X_{n_k, [2]}}^{(j_{n_k})}}{E[X_{n_k}^{j_{n_k}}]}\right).
	\end{align*}
\end{corollary}
The special case where $C$ is the independence copula yields $A_1(j_1, \dots, j_k) = A_2(j_1, \dots, j_k) = 1$ for $k = 2, \dots, d$.

\section{Aggregation of some continuous rvs}\label{sec:continuous}

This section investigates special cases of 
distributions for positive continuous rvs which are closed under convolution when the dependence structure is a FGM copula, or which admit closed-form solutions for the $m$th moments. The guiding principle is that when rvs have closed-form representations for (i) the cdfs of their order statistics, or (ii) the $m$th moments of their order statistics, then one may obtain equivalent results for the aggregate rvs.

\subsection{Implications with exponential marginals}

We have seen in Example \ref{ex:exp-order-stats} that when $X$ is exponentially distributed, then $X_{[1]}$ and $X_{[2]}$ have convenient stochastic forms. It isn't surprising, given the link between the FGM copula and order statistics, that FGM copulas are based on exponential FGM distributions, first studied in their namesake papers, \cite{eyraud1936principes}, \cite{farlie1960performance}, \cite{gumbel1960bivariate} and \cite{morgenstern1956einfache}. When $\boldsymbol{X}$ has cdf $F_{\boldsymbol{X}}(\boldsymbol{x}) = C(F_{X_1}(x_1), \dots, F_{X_d}(x_d))$ with $C \in \mathcal{C}_d^{FGM}$ and $F_{X_j}(x) = 1 - \exp\{-\beta_j x\}$, for $j\in\{1,\dots,d\}$, then $S$ will also have a convenient stochastic form. 

\subsubsection{Case with exponential marginals with identical parameters}

We now study the special case where $F_{X_k}(x) = F(x) = 1 - \exp\{-\beta x\}$, for $x \geq 0$ and $k = 1, \dots, d$. For notational purposes, we introduce the rv $N_d$ which corresponds to the sum of the components from $\boldsymbol{I}$, that is, $N_d = \sum_{k = 1}^{d}I_k$. It follows from Theorem \ref{prop:lst-s} that
\begin{align}
	\laplace{S}{t} &= E\left[\left(\frac{2\beta}{2\beta + t} \frac{\beta}{\beta+t}\right)^{N_d}\left(\frac{2\beta}{2\beta + t}\right)^{d - N_d}\right]\nonumber\\
	&= \left(\frac{2\beta}{2\beta + t}\right)^d E\left[\left(\frac{\beta}{\beta + t}\right)^{N_d}\right] = \left(\frac{2\beta}{2\beta + t}\right)^d \mathcal{P}_{N_d}\left(\frac{\beta}{\beta + t}\right),\label{eq:lst-id-exponential}
\end{align}
for $t \geq 0$, where $\mathcal{P}_J(t)$ is the probability generating function (pgf) of a discrete rv $J$. From the form of $\mathcal{L}_S$ in \eqref{eq:lst-id-exponential}, one recognizes that $S$ is the sum of two independent rvs $Y_1$ and $Y_2$, where $Y_1 \sim Erlang(d, 2\beta)$ and $Y_2$ follows a compound distribution with cdf
$$F_{Y_2}(x) = E\left[H(x, N_d, \beta)\right] = \sum_{k = 0}^d \Pr(N_d = k) H(x, k, \beta), \quad x \geq 0,$$
where $H(x, \alpha, \beta)$ is the cdf of an Erlang distribution with shape $\alpha$ and rate $\beta$, and with $H(x, 0, \beta) = 1$. We conclude that $Y_2$ follows a finite mixture of Erlang distributions with probabilities given by the pmf of $N_d$ and rate parameter $\beta$. 

Further, one can show that $S$ follows a mixed Erlang distribution. Following \cite{willmot2007class}, we write the LST of $Y_1$ and $Y_2$ under the same rate parameter using the identity
\begin{equation}\label{eq:lemma-beta}
	\frac{\beta_1}{\beta_1 + t} = \frac{\beta_2}{\beta_2 + t} \left\{ \frac{\beta_1 / \beta_2}{1 - (1 - \beta_1 / \beta_2) \frac{\beta_2}{\beta_2 + t}}\right\},
\end{equation}
for $0 < \beta_1 \leq \beta_2 < \infty$ and $t \geq 0$. Specifically, combining \eqref{eq:lst-id-exponential} and \eqref{eq:lemma-beta}, we obtain
\begin{equation}\label{eq:laplace-sum-exp}
	\laplace{S}{t} = \left(\frac{2\beta}{2\beta + t}\right)^d \mathcal{P}_{N_d}\left(\frac{2\beta}{2\beta + t} \left\{ \frac{0.5}{1 - 0.5 \frac{2\beta}{2\beta + t}}\right\}\right) = \mathcal{P}_M\left(\frac{2\beta}{2\beta + t}\right),
\end{equation}
where 
$$\mathcal{P}_{M}(t) = t^d \mathcal{P}_{N_d}\left(\frac{0.5t}{1 - 0.5t}\right), \quad t \geq 0.$$
From the expression in \eqref{eq:laplace-sum-exp}, we deduce that $S$ follows a mixed Erlang distribution with rate $2\beta$ and parameters $q_k = \Pr(M = k),$ for $k \in \mathbb{N}_1$.

\subsubsection{Case with exponential marginals with different parameters}\label{ss:non-id-exponential}

We now consider the case where $F_{X_k}(x) = 1 - \exp\{-\beta_k x\}, x \geq 0, k = 1, \dots, d$, and where $\beta_1 \neq \dots \neq \beta_d$. Applying Theorem \ref{prop:lst-s} to the order statistic representation of exponentially distributed rvs provided in Example \ref{ex:exp-order-stats}, the LST of $S$ is
\begin{equation}\label{eq:non-id-exp}
	\laplace{S}{t} = 
	E\left[\prod_{k = 1}^{d} \left(\frac{2\beta_k}{2\beta_k + t}\right)\left(\frac{\beta_k}{\beta_k + t}\right)^{I_k}\right] = \left\{\prod_{k = 1}^{d} \left(\frac{2\beta_k}{2\beta_k + t}\right)\right\} \times E\left[\prod_{k = 1}^{d} \left(\frac{\beta_k}{\beta_k + t}\right)^{I_k}\right], \quad t \geq 0.
\end{equation}
One may decompose the LST in \eqref{eq:non-id-exp} as the product of two LSTs, hence $S$ is the sum of two independent rvs that we denote $Y_1$ and $Y_2$. One observes that $Y_1$ follows a generalized Erlang distribution with cdf
$$F_{Y_1}(x) = \sum_{k = 1}^{d} \left( \prod_{j = 1, j \neq k}^{d} \frac{\beta_j}{\beta_j - \beta_k}\right)\left(1 - e^{-2\beta_j x}\right), \quad x \geq 0.$$
Since the conditional rv $Y_2 \vert \boldsymbol{I}$ also follows a generalized Erlang distributions, $Y_2$ is a finite mixture of generalized Erlang rvs with cdf
\begin{align*}
	F_{Y_2}(x) &= E_{\boldsymbol{I}}\left[F_{Y_2 \vert \boldsymbol{I}}(x)\right] = \sum_{\boldsymbol{i} \in \{0, 1\}^d} f_{\boldsymbol{I}}(\boldsymbol{i}) F_{Y_2 \vert \boldsymbol{I} = \boldsymbol{i}}(x)\\
		&= \sum_{\boldsymbol{i} \in \{0, 1\}^d} f_{\boldsymbol{I}}(\boldsymbol{i})\sum_{\{k \in \{1, \dots, d\}\vert i_k = 1\}} \left( \prod_{\{j \in \{1, \dots, d\} \vert  i_j = 1, j \neq k\}} \frac{\beta_j}{\beta_j - \beta_k}\right)\left(1 - e^{-\beta_j x}\right), \quad x \geq 0.
\end{align*}
Once again, we can show that $S$ follows a mixed Erlang distribution and use \eqref{eq:lemma-beta} to set all cdfs under the same rate parameter. 

\subsection{Mixed Erlang distributions}\label{ss:mx-erlang-sum}

Let $X_k$, for $k = 1, \dots, d$, follow mixed Erlang distributions, parametrized by vectors of probabilities $\{q_{k, j}, j \in \mathbb{N}_1\}$, a common rate parameter $\beta$, and cdfs
$$F_{X_k}(x) = \sum_{j = 1}^\infty q_{k, j} H(x; j, \beta), \quad x \geq 0.$$
Also, let $L_k$ be the discrete rv with pmf $\Pr(L_k = j) = q_{k, j}$, for $j \in \mathbb{N}_1$ and $k\in \{1,\dots,d\}$. Then, the LST of $X_k$ is given by $\mathcal{L}_{X_k}(t) = \mathcal{P}_{L_k}\left(\beta/(\beta + t)\right),$ for $k \in \{1, \dots, d\}$ and $t\geq 0$.

In \cite{landriault2015note}, the authors show that when a rv $X$ is mixed Erlang distributed, then $X_{[1]}$ and $X_{[2]}$ are also mixed Erlang distributed. We briefly recall their result and provide expressions for the new parameters of the distributions for the rvs $X_{[1]}$ and $X_{[2]}$. Let $Q_{k, j} = \sum_{m = 1}^j q_{k, m}$, for $j \in \mathbb{N}_1$ and $Q_{k, 0} = 0$. One has 
$$F_{X_{k, [i+1]}}(x) = \sum_{j = 1}^\infty q_{k, j, \{i+1\}} H(x; j, 2\beta), \quad k = 1, \dots, d,\quad i \in \{0, 1\}, \quad x > 0,$$
with 
\begin{equation}\label{eq:pmf-order-stat-mxerl}
	q_{k, j, \{i+1\}} = \begin{cases}
		\frac{1}{2^{j-1}} \sum_{m = 0}^{j-1} \binom{j-1}{m} q_{k, m + 1} \left(1 - Q_{k, j-1-m}\right), & \text{for } i = 0\\
		\frac{1}{2^{j-1}} \sum_{m = 0}^{j-1} \binom{j-1}{m} q_{j, m + 1} Q_{k, j-1-m}, & \text{for } i = 1
	\end{cases},
\end{equation}
for $j \in \mathbb{N}_1$, where \eqref{eq:pmf-order-stat-mxerl} is a special case of equation (2.7) from \cite{landriault2015note}. 

Note that $q_{k, j, \{i+1\}}$ does necessarily correspond to $\Pr(L_{k, [i+1]} = j)$, for $k\in\{1,\dots,d\}$ and $j\in\mathbb{N}_1$, it is for this reason that we use the braces notation instead of the brackets notation. We denote $L_{k, \{i+1\}}$ the rv with probability masses $\left\{q_{k, j, \{i+1\}}, j \in \mathbb{N}_1\right\}$, for $k \in \{1, \dots, d\}$ and $i \in \{0, 1\}$. 

From Theorem \ref{prop:lst-s}, the Laplace-Stieltjes transform of $S$ is 
\begin{equation*}
	\laplace{S}{t} = E_{\boldsymbol{I}}\left[\prod\limits_{k=1}^{d} \sum_{j = 1}^\infty q_{k, j, \{I_k+1\}}\left(\frac{2\beta}{2\beta + t}\right)^j\right] = E_{\boldsymbol{I}}\left[\prod\limits_{k=1}^{d} \mathcal{P}_{L_{k, \{I_k+1\}}}\left(\frac{2\beta}{2\beta + t}\right)\right], \quad t\geq 0.
\end{equation*}
Noticing that $S \vert \boldsymbol{I}$ is the sum of $d$ independent compound distributed rvs, we rearrange the LST as 
\begin{equation}\label{eq:lst-mxerl-int}
    \laplace{S}{t} = \sum_{\boldsymbol{i}\in \{0,1\}^d} f_{\boldsymbol{I}}(\boldsymbol{i}) \mathcal{P}_{M_{\boldsymbol{i}}}\left(\frac{2\beta}{2\beta + t}\right) = \sum_{\boldsymbol{i}\in \{0,1\}^d} f_{\boldsymbol{I}}(\boldsymbol{i}) \sum_{j = 1}^{\infty} \Pr(M_{\boldsymbol{i}} = j)\left(\frac{2\beta}{2\beta + t}\right)^j,
\end{equation}
where $M_{\boldsymbol{i}} = L_{1, \{i_1+1\}} + \dots + L_{d, \{i_d+1\}}$ for $\boldsymbol{i} \in \{0, 1\}^d$ and $t \geq 0$. We suggest using the fast Fourier transform of \cite{cooley1965algorithm} to compute the probability masses of $M_{\boldsymbol{i}}$. From the expression in \eqref{eq:lst-mxerl-int}, we conclude that $S$ also follows a mixed Erlang distribution. Indeed, we deduce from \eqref{eq:lst-mxerl-int} that
\begin{equation}\label{eq:cdf-mx-erl}
	F_S(x) = \sum_{\boldsymbol{i}\in \{0,1\}^d} f_{\boldsymbol{I}}(\boldsymbol{i})\sum_{j = 1}^\infty \Pr(M_{\boldsymbol{i}} = j) H(x; j, 2\beta) = \sum_{j = 1}^\infty q_{S, j} H(x; j, 2\beta), \quad x \geq 0,
\end{equation}
with 
\begin{equation}\label{eq:weights-mxerl}
	q_{S, j} = \sum_{\boldsymbol{i}\in \{0,1\}^d} f_{\boldsymbol{I}}(\boldsymbol{i}) \Pr(M_{\boldsymbol{i}} = j), \quad j \in \mathbb{N}_1.
\end{equation}
From \eqref{eq:weights-mxerl}, one may compute risk measures of the aggregate rv $S$. For instance, from \eqref{eq:cdf-mx-erl}, we have that the TVaR of $S$ at level $\kappa \in (0, 1)$ is 
\begin{equation}\label{key}
	\text{TVaR}_{\kappa}(S) = \sum_{j = 1}^\infty q_{S, j} \frac{j}{2\beta}\overline{H}(\text{VaR}_{\kappa}(S); j + 1, 2\beta),
\end{equation}
where $\overline{H}(x; \alpha, \beta) = 1 - H(x; \alpha, \beta)$, and where $\text{VaR}_{\kappa}(S)$ is obtained by numerical inversion of \eqref{eq:cdf-mx-erl}. 

The results from this subsection were previously shown, though stated differently, in Proposition 4.2 of \cite{cossette2013multivariate} from a purely algebraic argument, under the natural representation of the FGM copula. The significant contribution from this subsection is that formulas are much simpler and more intuitive. In addition, the stochastic representation of FGM copulas break down the problem of computing \eqref{eq:weights-mxerl} as a convolution or a mixture of the discrete probability masses in \eqref{eq:pmf-order-stat-mxerl}. From a programming standpoint, this is an important advantage since one can validate proper computation of pmfs at all intermediate steps. Finally, one can obtain similar results for mixed Erlang distributions that do not share the same rate parameter using the same strategy as the one used in Section \ref{ss:non-id-exponential}, see also Section 2 of \cite{willmot2007class}.

So far, we showed that when the random vector $\boldsymbol{X}$ has mixed Erlang marginals, and when the copula defining $F_{\boldsymbol{X}}$ is FGM, then the aggregate rv is also mixed Erlang distributed. The conditions for this result are that each marginal distribution are closed under order statistics, finite mixture and convolution (for each marginal and across the random vector). As shown in \cite{bladt2017matrixexponential}, phase-type and matrix-exponential distributions also have the three closure properties. If follows that the aggregate rv of phase-type and matrix-exponential distributions under FGM dependence will also respectively follow phase-type and matrix-exponential distributions, though we defer investigating the implications of these statements to future research.

\subsection{Some closed-form moments of continuous rvs}

We now study random vectors with Pareto and Weibull marginals, and find closed-form expressions for the $m$th moments of the aggregate rv. Since one knows the distribution of $X_{[1]}$ when $X$ is Pareto or Weibull distributed, then one may compute the moments of $X$, of $X_{[1]}$ and of $X_{[2]}$, the latter requiring the identity in \eqref{eq:order-statistic-moment-relationship}. Note that since Pareto and Weibull distributions do not have convenient closure properties for mixtures and convolution operations, we may not obtain exact expressions for the cdf of $S$ as we have for mixed Erlang distributions. 

First assume that the rv $X$ follows a Pareto type IV distribution, denoted Pareto(IV), with survival function
$$\overline{F}_{X}(x) =  \left[1 + \left(\frac{x - \mu}{\sigma}\right)^{1/\gamma}\right]^{-\alpha}, \quad x > \mu,$$
with $\mu \in \mathbb{R}$ and $\sigma, \gamma, \alpha > 0$. One obtains the Lomax distribution, popular in actuarial science by setting $\mu = 0$ and $\gamma = 1$. The case $\mu = 0$ simplifies to a Burr type XII distribution, while $\alpha = 1$, simplifies to the log-logistic distribution. See \cite{arnold2015pareto} for more details. When $X$ follows a Pareto(IV), its m$th$ moment exists for $-\gamma ^{-1} < m < \alpha / \gamma$ and is given by
\begin{equation}\label{eq:moment-m-pareto}
	E[X^m] = \sigma^m \frac{\Gamma(\alpha - \gamma m)\Gamma(1 + \gamma m)}{\Gamma(\alpha)}.
\end{equation}
The survival function for $X_{[1]}$ when $X$ follows a Pareto(IV) distribution is
$$\overline{F}_{X_{[1]}}(x) =  \left[1 + \left(\frac{x - \mu}{\sigma}\right)^{1/\gamma}\right]^{-2\alpha}, \quad x > \mu,$$
which is the survival function of a Pareto(IV) distribution but with parameter $2\alpha$. Therefore, for $-\gamma^{-1} < m < 2\alpha / \gamma$, we have
\begin{equation}\label{eq:moment-m-pareto-min}
	\mu_{X_{[1]}}^{(m)} = \sigma^m \frac{\Gamma(2\alpha - \gamma m)\Gamma(1 + \gamma m)}{\Gamma(2\alpha)}.
\end{equation}
Inserting \eqref{eq:moment-m-pareto} and \eqref{eq:moment-m-pareto-min} within \eqref{eq:m-moment-s-stochastic-min}, we have, for $-\gamma^{-1} < m < \alpha / \gamma$, that
\begin{align*}
	E\left[S^m\right] &= \sum_{j_1 + \cdots + j_d = m} \frac{m!}{j_1!\cdots j_d!} \left\{\prod_{k = 1}^{d}\sigma^{j_k} \frac{\Gamma(\alpha - \gamma j_k)\Gamma(1 + \gamma j_k)}{\Gamma(\alpha)}\right\} \\
	&\qquad \times E_{\boldsymbol{I}}\left[\prod_{k = 1}^{d}\left\{1 + (-1)^{I_k}\left(\frac{\Gamma(2\alpha - \gamma j_{k})}{\Gamma(\alpha - \gamma j_k)}\frac{2^{1-2\alpha} \sqrt{\pi}}{\Gamma(\alpha + 1/2)} - 1\right)\right\}\right].
\end{align*}
Alternatively from \eqref{eq:moments-natural-formulation}, we have 
\begin{align*}
	E\left[S^m\right] &= \sum_{j_1 + \cdots + j_d = m} \frac{m!}{j_1!\cdots j_d!} \left\{\prod_{k = 1}^{d}\sigma^m \frac{\Gamma(\alpha - \gamma j_k)\Gamma(1 + \gamma j_k)}{\Gamma(\alpha)}\right\}  \\
	& \qquad \times\left(1 + \sum_{k = 2}^{d} \sum_{1 \leq n_1 < \cdots < n_k \leq d}\theta_{n_1\cdots n_k} \left(\frac{\Gamma(2\alpha - \gamma j_{n_1})}{\Gamma(\alpha - \gamma j_{n_1})}\frac{2^{1-2\alpha} \sqrt{\pi}}{\Gamma(\alpha + 1/2)} - 1\right)\right. \\
	& \qquad\qquad \left.\times \dots \times \left(\frac{\Gamma(2\alpha - \gamma j_{n_k})}{\Gamma(\alpha - \gamma j_{n_k})}\frac{2^{1-2\alpha} \sqrt{\pi}}{\Gamma(\alpha + 1/2)} - 1\right) \vphantom{\sum_{k = 2}^{d}}\right).
\end{align*}

Next, assume that $X$ follows a Weibull distribution with pdf
$$f_{X}(x) = \beta \tau (\beta x)^{\tau - 1} e^{-(\beta x)^{\tau}}, \quad x \geq 0,$$
and survival function
$$\overline{F}_{X}(x) = e^{-(\beta x)^{\tau}}, \quad x \geq 0,$$
where $\beta, \tau > 0$, with moments given by $E[X^m] = \beta^{-m}\Gamma\left(1 + m/\tau\right)$. One computes
\begin{align*}
	\mu_{X_{[1]}}^{(m)} &= 2 \int_{0}^{\infty} x^{m}(1-e^{-(\beta x)^{\tau}})\beta \tau (\beta x)^{\tau - 1} e^{-(\beta x)^{\tau}} \diff x = 2 E[X^m] - 2 \int_{0}^{\infty} \beta^{\tau} \tau x^{\tau + m - 1} e^{-2\beta^{\tau} x^{\tau}} \diff x.
\end{align*}
Letting $u = x^{\tau}$, we find 
\begin{equation}\label{eq:weibull-mean-min}
	\mu_{X_{[1]}}^{(k)} = 2 E[X^k] - \int_{0}^{\infty} 2\beta^{\tau}  u^{m/\tau} e^{-2\beta^{\tau} u} \diff u
	= \frac{1}{\beta^m}\Gamma\left(1 + \frac{m}{\tau}\right)\left(2 - 2^{-m/\tau}\right).
\end{equation}
Inserting \eqref{eq:weibull-mean-min} within \eqref{eq:m-moment-s-stochastic-min} or \eqref{eq:moments-natural-formulation}, we have respectively
\begin{align}
	E\left[S^m\right] &= \sum_{j_1 + \cdots + j_d = m} \frac{m!}{j_1!\cdots j_d!} \left\{\prod_{k = 1}^{d}\frac{1}{\beta^{j_k}}\Gamma\left(1 + \frac{j_k}{\tau}\right)\right\} E_{\boldsymbol{I}}\left[\prod_{k = 1}^{d}\left\{1 + (-1)^{I_k}\left(1 - 2^{-j_k/\tau}\right)\right\}\right]\label{eq:moment-weibull-stochastic}\\
	&= \sum_{j_1 + \cdots + j_d = m} \frac{m!}{j_1!\cdots j_d!} \left\{\prod_{k = 1}^{d}\frac{1}{\beta^{j_k}}\Gamma\left(1 + \frac{j_k}{\tau}\right)\right\}\nonumber \\
	& \qquad \times \left(1 + \sum_{k = 2}^{d} \sum_{1 \leq n_1 < \cdots < n_k \leq d}\theta_{n_1\cdots n_k} \left(1 - 2^{-j_{n_1}/\tau}\right) \cdots \left(1 - 2^{-j_{n_k}/\tau}\right)\right).\label{eq:moment-weibull-natural}
\end{align}
Note that in Section 5.6.3 of \cite{kotz2001correlation}, the authors develop an expression similar to \eqref{eq:moment-weibull-stochastic} for the product moments of $\boldsymbol{X}$ under the natural representation of the FGM copula for Weibull marginals. The advantage of the approach we take in the current paper is that one obtains the result by directly applying Corollary \eqref{cor:moments}, and this corollary holds for any combination of marginal distributions. 

\subsection{Closing remarks for continuous rvs}

In this section, we showed that if $\boldsymbol{X}$ has joint cdf $F_{\boldsymbol{X}}(\boldsymbol{x}) = C(F_{X_1}(x_1), \dots, F_{X_d}(x_d))$, for $\boldsymbol{x} \in \mathbb{R}^d_+$, where $C \in \mathcal{C}^{FGM}_d$ and $F_{X_k}, k =1, \dots, d$, are the cdfs of mixed Erlang rvs, then $S$ follows a mixed Erlang distribution. Also, when $F_{X_k}$ is the cdf of Pareto(IV) or Weibull distributions, then we have closed-form expressions for the $m$th moments of $S$. We close this section by discussing other distributions which could admit convenient expressions with FGM copulas, and why we have not considered them. 

The relationship in \eqref{eq:m-moment-s-stochastic-min}, which uses the moments of the minimum order statistic, is usually more useful: for survival functions defined as compositions of a first function with a power functon, the survival function of the minimum will also be defined as compositions of a first function with another power function. That is, squaring a power function will yield another power function. This is the reason why we obtain closed-form expressions for Pareto(IV) and Weibull marginals. Another example which satisfies this condition is when $X$ follows a Gompertz-Makeham distribution, then $X_{[1]}$ also follows a Gompertz-Makeham distribution. However, contrarily to Pareto and Weibull distributions, computing the moments from Gompertz-Makeham distributions require numerical integration. 

On the other hand, when the cdf is defined as the composition of a first function with a power function, then squaring the cdf will yield a cdf in the same family as the original cdf, so the moment associated with $X_{[2]}$ has a preferable shape for computations. A simple example is the standard power function distribution with cdf $F_X(x) = x^{\alpha}, x \in [0, 1], \alpha > 0$. Another example is the Gumbel distribution with cdf
$$F_{X}(x) = \exp\left\{-\exp\left(-\frac{x - \mu}{\sigma}\right)\right\}, \quad x \in \mathbb{R},$$
with $\mu \in \mathbb{R}, \sigma \in \mathbb{R}^+$. Then, the cdf of the maximum is also the cdf of a Gumbel distribution:
$$F_{X_{[2]}}(x) = \exp\left\{-2\exp\left(-\frac{x - \mu}{\sigma}\right)\right\} = \exp\left\{-\exp\left(-\frac{x - \mu}{\sigma/\ln 2}\right)\right\}, \quad x \in \mathbb{R}.$$
However, the moment generating function of a Gumbel distribution being $\Gamma(1 - \sigma t)\exp(\mu t)$, moments are tedious to compute. 

In \cite{nadarajah2008explicit}, the author presents expressions for the moments of order statistics for normal and log-normal distributions. These expressions are a function of a finite sum of Lauricella functions of type A. The moments of order statistics for log-normal distributions still require numerical integration. Since the expressions for these moments are tedious, we omit them in the current paper. 

Finally, we note that \eqref{eq:m-moment-s-stochastic-min}, \eqref{eq:m-moment-s-stochastic-max} and \eqref{eq:moments-natural-formulation} do not require to assume the same marginal distributions; one can compute the exact value of a given $m$th moment for the sum of a combination of, for example, mixed Erlang, phase-type, matrix-exponential, Pareto(IV) and Weibull distributions with different parameters, provided the $j$th moments, $j = 1, \dots, m$ of the maximum of each marginal distribution is finite. 

\section{Studying the impact of dependence with stochastic orders}\label{sec:stochastic-order}

In this section, we will leverage the stochastic representation of FGM copulas to study the impact of dependence on the aggregate rv $S$. We briefly recall the notions required for this section. Let $\boldsymbol{X}$ and $\boldsymbol{X}'$ be two $d$-variate random vectors whose cdfs belong to the same Fréchet class $\mathcal{FC}(F_{X_1}, \dots, F_{X_d})$. Our aim is to compare the rvs $S$ and $S'$, which respectively correspond to the sum of rvs from the random vectors $\boldsymbol{X}$ and $\boldsymbol{X}'$. An important stochastic order in actuarial science, which measures the variability of a rv, is the convex order.
\begin{definition}[Convex order]
	Let $Y$ and $Y'$ be two rvs with finite expectations. We say that $Y$ is smaller than $Y'$ under the convex order if $E[\phi(Y)] \leq E[\phi(Y')]$ for every convex function $\phi$, when the expectations exist. We denote two rvs ordered according to the convex order as $Y \preceq_{cx} Y'$.
\end{definition}
Some relevant implications of the relation $S \preceq_{cx}S'$ are that $E[S] = E[S']$, $Var(S) \leq Var(S')$ (assuming that they exist), and $\text{TVaR}_{\kappa}(S) \leq \text{TVaR}_{\kappa}(S')$, for all $\kappa \in (0, 1)$, see \cite{muller2002comparison,denuit2006actuarial, shaked2007stochastic} for a more comprehensive list. 

In our quest to order the aggregate rvs according to the convex order, we will first need to compare vectors of rvs, $(V_{1},\dots,V_{d})$ and $(V_{1}^{\prime}, \dots, V_{d}^{^{\prime}})$, using dependence stochastic orders, where, for each $j \in \{1,\dots,d\}$, $V_{j}$ and $V_{j}'$ have the same marginal distribution. In Sections 3.8 and 3.9 of \cite{muller2002comparison}, the authors present the supermodular order.
\begin{definition}[Supermodular order] \label{defSupermodularOrder}
	We say $\boldsymbol{V}$ is smaller than $\boldsymbol{V}'$ under the supermodular order, denoted $\boldsymbol{V}\preceq_{sm}\boldsymbol{V}'$, if $E\left[\phi (\boldsymbol{V})\right] \leq E\left[\phi (\boldsymbol{V}')\right]$ for all supermodular functions $\phi $, given that the expectations exist. A function $\phi :\mathbb{R}^{d}\rightarrow \mathbb{R}$ is said to be supermodular if
	\begin{eqnarray*}
		&&\phi (x_{1},\ldots,x_{i}+\varepsilon ,\ldots,x_{j}+\delta ,\ldots,x_{d})-\phi
		(x_{1},\ldots,x_{i}+\varepsilon ,\ldots,x_{j},\ldots,x_{d}) \\
		&\geq &\phi (x_{1},\ldots,x_{i},\ldots,x_{j}+\delta ,\ldots,x_{d})-\phi
		(x_{1},\ldots,x_{i},\ldots,x_{j},\ldots,x_{d})
	\end{eqnarray*}
	holds for all $(x_1, \dots, x_d)\in \mathbb{R}^{d}$, $1\leq
	i < j\leq d$\ and all $\varepsilon$, $\delta >0$.
\end{definition}
The supermodular order satisfies the nine desired properties for dependence orders as mentioned in Section 3.8 of \cite{muller2002comparison}. See also \cite{shaked2007stochastic} and \cite{denuit2006actuarial} for more details on the supermodular order. We recall the following Theorem from \cite{blier-wong2022stochastic} which presents the general result for supermodular orders within the family of FGM copulas. 
\begin{theorem}\label{thm:supermodular-i-u}
	Let $\boldsymbol{I}$ and $\boldsymbol{I}'$ be random vectors with $F_{\boldsymbol{I}}, F_{\boldsymbol{I}'} \in \mathcal{B}_d$. Let $\boldsymbol{U}$ and $\boldsymbol{U}'$ be random vectors constructed using \eqref{eq:U-based-on-I} and
	$\boldsymbol{X}$ and $\boldsymbol{X}'$ be random vectors constructed using \eqref{eq:X-based-on-I}. If $\boldsymbol{I} \preceq_{sm} \boldsymbol{I}'$, then $\boldsymbol{U} \preceq_{sm} \boldsymbol{U}'$ and $\boldsymbol{X} \preceq_{sm} \boldsymbol{X}'$.
\end{theorem}
Establishing the supermodular order within a class of copulas has important consequences for risk aggregation, as the following proposition shows.
\begin{proposition}\label{prop:order-cx-s}
	If $\boldsymbol{X} \preceq_{sm} \boldsymbol{X}'$ holds, then $\sum_{j=1}^d X_j \preceq_{cx} \sum_{j=1}^d X_j'$, where $\preceq_{cx}$ is the convex order. 
\end{proposition}
\begin{proof}
	See Theorem 8.3.3 of \cite{muller2002comparison} or Proposition 6.3.9 of \cite{denuit2006actuarial}. 
\end{proof}
It follows from Proposition \ref{prop:order-cx-s} and Theorem \ref{thm:supermodular-i-u} that one may order the aggregate rvs $S$ and $S'$ within the context of the current paper if one first orders the random vectors $\boldsymbol{X}$ and $\boldsymbol{X}'$. In the remainder of this section, we investigate the implications of this fact. 

While the upper bound under the supermodular order for multivariate Bernoulli random vectors is well-known (see the EPD FGM copula further in this section), its lower bound is still an open problem. For this reason, we will restrict our analysis to the class of exchangeable FGM copulas, studied in \cite{blier2021exchangeable}, for which a lower bound exists. The lower and upper bounds of the supermodular order within the families of exchangeable FGM copulas, called respectively the extreme negative dependence (END) and extreme positive dependence (EPD) satisfy 
$$\boldsymbol{U}^{END} \preceq_{sm} \boldsymbol{U} \preceq_{sm} \boldsymbol{U}^{EPD},$$
for all $\boldsymbol{U}$ with $F_{\boldsymbol{U}}$ being an exchangeable FGM copula as defined in \cite{blier2021exchangeable}. Further, $\boldsymbol{U} \preceq_{sm} \boldsymbol{U}^{EPD}$ holds for all $\boldsymbol{U}$ with $F_{\boldsymbol{U}} \in \mathcal{C}^{FGM}$. We recall the definition of the EPD FGM copula from Theorem 5 of \cite{blier-wong2022stochastic}.
\begin{theorem}\label{thm:epd}
	The FGM copula associated to the random vector $\boldsymbol{I}$ whose components are comonotonic rvs is the EPD FGM copula, denoted by $C^{EPD}$. The expression of the EPD FGM copula is given by 
	\begin{equation} \label{eq:epd}
		C^{EPD}\left(\boldsymbol{u}\right) = \prod_{k = 1}^{d} u_k \left(1 + \sum_{k = 1}^{\left\lfloor \frac{d}{2} \right\rfloor}\sum_{1\leq j_{1}<\cdots <j_{2k}\leq d} \overline{u}_{j_1}\cdots \overline{u}_{j_{2k}}\right), \quad \boldsymbol{u} \in [0,1]^{d},
	\end{equation}
	where $\lfloor y \rfloor$ is the floor function returning the greatest integer smaller or equal to $y$. The $k$-dependence parameters are $\theta_{k} = (1 + (-1)^k)/2$, for $k \in \{2, \dots, d\}$.  
\end{theorem}

The END FGM copula is derived in \cite{blier2021exchangeable} and recalled in the following theorem.
\begin{theorem}
	The expression of the FGM END copula, denoted by $C^{END}$, is given by 
	\begin{equation} \label{eq:end}
		C^{END}\left(\boldsymbol{u}\right) = \prod_{k = 1}^{d} u_k \left(1 + \sum_{k = 1}^{\left\lfloor \frac{d}{2} \right\rfloor}\sum_{1\leq j_{1}<\cdots <j_{2k}\leq d}
		\frac{\Gamma(k+1)\Gamma\left(\frac{1}{2} -\left\lfloor  \frac{d + 1}{2}\right\rfloor\right)}{2^k \Gamma\left(\frac{k}{2}+1\right) \Gamma\left(\frac{k+1}{2} - \left\lfloor \frac{d + 1}{2}\right\rfloor\right)}
		\overline{u}_{j_1}\cdots \overline{u}_{j_{2k}}\right), \quad \boldsymbol{u} \in [0,1]^{d}.
	\end{equation}
	That is, the $k$-dependence parameters for the FGM END copula are given by
	\begin{equation}\label{eq:end-param}
		\theta_k = {}_2F_1\left(-\left\lfloor \frac{d + 1}{2}\right\rfloor, -k, 2\left\lfloor \frac{d + 1}{2}\right\rfloor, 2\right) = \frac{(1 + (-1)^k)}{2}\frac{\Gamma(k+1)\Gamma\left(\frac{1}{2} -\left\lfloor  \frac{d + 1}{2}\right\rfloor\right)}{2^k \Gamma\left(\frac{k}{2}+1\right) \Gamma\left(\frac{k+1}{2} - \left\lfloor \frac{d + 1}{2}\right\rfloor\right)},
	\end{equation}
	for $k \in \{2, \dots, d\}$ and where ${}_2F_1$ is the ordinary hypergeometric function. 
\end{theorem}

\begin{example}
	Consider a vector $\boldsymbol{X}$ with joint cdf $F_{\boldsymbol{X}}(\boldsymbol{x}) = C(F(x_1), \dots, F(x_d))$, $\boldsymbol{x} \in \mathbb{R}^d_+$, where $F(x) = 1 - e^{-\beta x}$ and $C$ is an exchangeable FGM copula. We denote the aggregate rv for a portfolio of $d$ risks as $S_d$, and omit the subscript when $d$ is arbitrary. In this example, we study the special cases of $S$ which lead to the lower bound and the upper bound under the convex order for exchangeable FGM copulas, respectively denoted $S^{-}$ and $S^{+}$, along with the aggregate rv under the assumption of independence, denoted $S^{\perp}$. Note that $E[S^{-}] = E[S] = E[S^{+}]$, $Var(S^{-}) \leq Var(S) \leq Var(S^{+})$ and $\mathrm{TVaR}_{\kappa}(S^{-}) \leq \mathrm{TVaR}_{\kappa}(S) \leq \mathrm{TVaR}_{\kappa}(S^{+})$, for all $\kappa \in (0, 1)$ and for all $S$ constructed within the setup of this example. By using the representation in \eqref{eq:S-representation-sum} and Theorem \ref{prop:lst-s}, the LST of $S^+$ is
	$$\laplace{S_d^+}{t} = \frac{1}{2}\left(\frac{2\beta}{2\beta +t}\right)^d + \frac{1}{2}\left(\frac{\beta}{\beta +t}\frac{2\beta}{2\beta +t}\right)^d, \quad t \geq 0,$$
	while the LST of $S^{-}$ is
	$$\laplace{S_d^-}{t} = \begin{cases}
		\left(\frac{2\beta}{2\beta + t}\right)^d \left(\frac{\beta}{\beta + t}\right)^{d/2}, & d \text{ is even}\\
		\left(\frac{2\beta}{2\beta + t}\right)^d \left(\frac{1}{2}\left(\frac{\beta}{\beta + t}\right)^{(d-1)/2} + \frac{1}{2}\left(\frac{\beta}{\beta + t}\right)^{(d+1)/2}\right), & d \text{ is odd}\\
	\end{cases}.$$
	Both LSTs correspond to the LST of mixed Erlang distributions. Using an optimization tool, we obtain the values for the VaR, then we compute the TVaR. To simplify comparisons, we introduce the rv $W_d = S_d / d$. We present the values of VaR and TVaR for $W_d$ in Table \ref{tab:var-tvar-cv}. We present the results for $d\in \{1, 2, 10, 100, 1000\}$ and $\kappa \in \{0.9, 0.99, 0.999\}$. We compute every risk measure with $\beta = 0.1$, that is, $E[X] = 10$, $E[S_d] = 10\times d$ and $E[W_d] = 10$. 
	\begin{table}[ht]
		\centering
		\caption{VaR and TVaR of $W_d$ with the END, independent and EPD copulas.}\label{tab:var-tvar-cv}
	\begin{tabular}{lrrrrrrrrr}
		& \multicolumn{3}{c}{$\kappa = $ 0.9} & \multicolumn{3}{c}{$\kappa = $ 0.99} & \multicolumn{3}{c}{$\kappa = $ 0.999} \\
		\cmidrule(r{4pt}){2-4} \cmidrule(l){5-7}\cmidrule(l){8-10} 
		                                   & $END$ & $Ind$ & $EPD$ & $END$ & $Ind$ & $EPD$ & $END$ & $Ind$ & $EPD$ \\ \midrule
		$\text{VaR}\left(W_{1}\right)$     & 23.03 & 23.03 & 23.03 & 46.05 & 46.05 & 46.05 & 69.08 & 69.08 & 69.08 \\
		$\text{TVaR}\left(W_{1}\right)$    & 33.03 & 33.03 & 33.03 & 56.05 & 56.05 & 56.05 & 79.08 & 79.08 & 79.08 \\ \hline
		$\text{VaR}\left(W_{2}\right)$     & 18.09 & 19.45 & 20.90 & 29.91 & 33.19 & 35.55 & 41.46 & 46.17 & 48.86 \\ 
		$\text{TVaR}\left(W_{2}\right)$    & 23.25 & 25.47 & 27.37 & 34.93 & 38.85 & 41.36 & 46.47 & 51.66 & 54.43 \\ \hline
		$\text{VaR}\left(W_{10}\right)$    & 13.63 & 14.21 & 17.85 & 17.58 & 18.78 & 23.19 & 20.95 & 22.66 & 27.40 \\
		$\text{TVaR}\left(W_{10}\right)$   & 15.38 & 16.24 & 20.26 & 19.06 & 20.48 & 25.05 & 22.31 & 24.20 & 29.04 \\ \hline
		$\text{VaR}\left(W_{100}\right)$   & 11.13 & 11.30 & 15.93 & 12.14 & 12.47 & 17.39 & 12.92 & 13.38 & 18.44 \\
		$\text{TVaR}\left(W_{100}\right)$  & 11.58 & 11.83 & 16.60 & 12.48 & 12.87 & 17.86 & 13.21 & 13.72 & 18.82 \\ \hline
		$\text{VaR}\left(W_{1000}\right)$  & 10.35 & 10.41 & 15.30 & 10.65 & 10.75 & 15.74 & 10.87 & 11.01 & 16.04 \\
		$\text{TVaR}\left(W_{1000}\right)$ & 10.49 & 10.56 & 15.50 & 10.75 & 10.86 & 15.87 & 10.95 & 11.10 & 16.15 \\ \hline
	\end{tabular}
	\end{table}
	Let us examine the effect of dependence on the risk measures. We note the TVaR for dependence structure $m$ as $TVaR^{m}_{\kappa}(W_{d})$, for $m \in \{END, Ind, EPD\}$ and $\kappa = 0.9$. We aim to compute the relative effect of dependence for different portfolio sizes. For $d = 2$, we have 
	$$
	\frac{\mathrm{TVaR}^{END}_{0.9}(W_2) - \mathrm{TVaR}^{Ind}_{0.9}(W_2)}{\mathrm{TVaR}^{Ind}_{0.9}(W_2)} = -0.0870;
	\quad
	\frac{\mathrm{TVaR}^{EPD}_{0.9}(W_2) - \mathrm{TVaR}^{Ind}_{0.9}(W_2)}{\mathrm{TVaR}^{Ind}_{0.9}(W_2)} = 0.0744,$$ 
	while, for $d = 1000,$ we have 
	$$
	\frac{\mathrm{TVaR}^{END}_{0.9}(W_{1000}) - \mathrm{TVaR}^{Ind}_{0.9}(W_{1000})}{\mathrm{TVaR}^{Ind}_{0.9}(W_{1000})} =  -0.0072;
	\quad
	\frac{\mathrm{TVaR}^{EPD}_{0.9}(W_{1000}) - \mathrm{TVaR}^{Ind}_{0.9}(W_{1000})}{\mathrm{TVaR}^{Ind}_{0.9}(W_{1000})} = 0.4671.$$ 
	The most negative relative effect of dependence appears for $d = 2$ with the END copula and decreases as the portfolio size $d$ increases. This isn't surprising, as the impact of the negative dependence structure decreases when the dimension $d$ increases, see \cite{blier2021exchangeable}. The most positive relative effect of dependence appears for the EPD copula and is an increasing function of $d$, that is, increasing the portfolio size with the EPD copula increases the relative effect of dependence on the TVaR. 
\end{example}

\section{Sum of discrete rvs}\label{sec:discrete}

This section deals with discrete rvs. We first provide preliminary results for the order statistics of discrete rvs, then present an efficient algorithm to compute the pmf of the aggregate rv. Let $\boldsymbol{X} = (X_1, \dots, X_d)$ be a vector of discrete rvs whose joint cdf is defined with a FGM copula. We note $p_{k, j} = \Pr(X_k = j)$, for $k\in \{1, \dots, d\}$ and $j\in \mathbb{N}_0$. The pmf of the minimum and maximum of two iid discrete rvs are 
\begin{align}
	p_{k, j, [1]} &:= \Pr(\min(X_{k, 1}, X_{k, 2}) = j) = 2p_j \sum_{m = j + 1}^{\infty}p_m  + p_j^2= 2p_j \left(1 - \sum_{m = 0}^{j}p_m\right)+ p_j^2;\label{eq:pmf-min}\\
	p_{k, j, [2]} &:= \Pr(\max(X_{k, 1}, X_{k, 2}) = j) = 2p_j \sum_{m = 0}^{j-1}p_m + p_j^2 =2p_j \sum_{m = 0}^{j}p_m - p_j^2. \label{eq:pmf-max}
\end{align}
The identity $p_{j, k} = \left(p_{k, j, [1]} + p_{k, j, [2]}\right)/2$ holds for all $k\in\{1, \dots, d\}$ and $j\in\mathbb{N}_0$. Few discrete distributions admit neat representations for their order statistics; we illustrate this point with two examples. 
\begin{example}\label{ex:geom}
	Let $X \sim Geom(q)$ with $\Pr(X = j) = q(1-q)^{j}$ and $\Pr(X > j) = (1-q)^{j+1}$, for $j \in \mathbb{N}_0$. Then, $\Pr(X_{[1]} > j) = (1 - q)^{2j+2}$ and we conclude that $X_{[1]} \sim Geom(1 - (1-q)^2)$. One has $\Pr(X_{[2]} \leq j) = 1 - 2(1-q)^{j+1} - (1-q)^{2j+2}$. Letting $q^* = 1 - (1-q)^2$, it follows that 
	\begin{align*}
		\mathcal{P}_{X_{[i]}}(t) = (-1)^{i-1}\left(\frac{q^*}{1 - (1-q^*)t} - \frac{q}{1 - (1-q)t}\right) + \frac{q}{1 - (1-q)t}, \quad i \in \{1, 2\}, 
	\end{align*}
	for $|t| \leq 1$. Consider two identically distributed rvs $X_1$ and $X_2$ which follow geometric distributions, where $F_{X_1, X_2}(x_1, x_2) = C(F_{X_1}(x_1), F_{X_2}(x_2))$, for $(x_1, x_2) \in \mathbb{N}_0^2$ and $C \in \mathcal{C}_2^{FGM}$. Then, the pgf of $S = X_1 + X_2$ is 
	\begin{align*}
		\mathcal{P}_{S}(t) &= E_{I}\left[\prod_{k = 1}^{2}\left\{(-1)^{I_k}\left(\frac{q^*}{1 - (1-q^*)t} - \frac{q}{1 - (1-q)t}\right) + \frac{q}{1 - (1-q)t}\right\}\right]\\
		&= \left(\frac{q}{1 - (1-q)t}\right)^2 + \theta_{12} \left\{\frac{q^*}{1 - (1-q^*)t} - \frac{q}{1 - (1-q)t}\right\}^2\\
		&= (1+\theta_{12})\left(\frac{q}{1 - (1-q)t}\right)^2 + \theta_{12}\left(\frac{q^*}{1 - (1-q^*)t}\right)^2 - 2\theta_{12}\frac{q^*}{1 - (1-q^*)t}\frac{q}{1 - (1-q)t},
	\end{align*}
	which is the pgf of a mixture of three rvs: the first two follow negative binomial distributions and the third one is the sum of two independent geometric rvs with different success probabilities. Since there are no simple formulas for the pmf of the third rv, we do not have a simple formula for the pmf of $S$, although one may show that $S$ follows a mixture of Pascal distributions (studied in, for instance, \cite{furman2007convolution, mi2008some, zhao2010ordering, badescu2015modeling}) and, more generally, a matrix-geometric distribution (see \cite{bladt2017matrixexponential}).	
\end{example}
\begin{example}\label{ex:poisson}
	When $X \sim Pois(\lambda)$, \eqref{eq:pmf-min} and \eqref{eq:pmf-max} respectively become
	\begin{align*}
		p_{k, [1]} &= \frac{2\lambda^{k} e^{-\lambda}}{{k}!}\left\{1 - \frac{\Gamma(k + 1, \lambda)}{k !}\right\} + \left(\frac{\lambda^k e^{-\lambda}}{k!}\right)^2;\\
		p_{k, [2]} &= \frac{2\lambda^{k} e^{-\lambda}}{{k}!}\frac{\Gamma(k, \lambda)}{(k - 1)!} \times 1_{\{k \geq 1\}}+ \left(\frac{\lambda^k e^{-\lambda}}{k!}\right)^2,
	\end{align*}
	for $k \in \mathbb{N}_0$, where $\Gamma(x, \lambda)$ is the upper incomplete Gamma function, that is, 
	$\Gamma(x, \lambda) = \int_{x}^{\infty} t^{x-1}e^{-t} \diff t.$ There does not seem to have elegant representations for the sum of rvs $X_{[1]}$ and $X_{[2]}$ when $X$ follows a Poisson distribution. 
	
%
%
%
%
%
%
%
%
%
\end{example}

Examples \ref{ex:geom} and \ref{ex:poisson} show that convenient forms for the pmfs of order statistics for discrete rvs aren't trivial. However, one can still compute the exact values of the pmf of $S$. Using the same arguments as in the proof of Theorem \ref{prop:lst-s}, the pgf of $S$ for discrete marginal rvs and a dependence structure induced by a FGM copula is 
\begin{equation}\label{eq:pgf-s}
	\mathcal{P}_{S}(t) = E_{\boldsymbol{I}}\left[\prod_{k = 1}^d \mathcal{P}_{X_{k, [I_k + 1]}}(t)\right].
\end{equation}
It follows that the representation in \eqref{eq:pgf-s} enables an algorithmic approach to find the pmf of $S$ for discrete rvs. Suppose there is a number $\omega \in \mathbb{N}_0$ such that $p_{k, \ell} = 0$, for all $k \in \{1, \dots, d\}$ and $\ell \geq \omega$. The discrete Fourier transform of $S$ forms a vector $\boldsymbol{\phi}_{S}$ with elements 
$\phi_{S, j} = \mathcal{P}_{S}(\exp\{-2\pi \mathrm{i} j/(d \times \omega)\})$, for $j = 0, \dots, d \times \omega - 1$. Therefore, the values of the pmf of $S$ are given by
\begin{equation}\label{eq:pmf-s-fft}
	p_{S, j} = \frac{1}{d \times \omega} \sum_{k = 0}^{d \times \omega - 1}E_{\boldsymbol{I}}\left[\prod_{n = 1}^d \mathcal{P}_{X_{n, [I_n + 1]}}\left(  \exp\{-2\pi \mathrm{i} k/(d \times \omega)\}   \right)\right] \exp\{2\pi \mathrm{i} kj / (d \times \omega)\},
\end{equation}
for $j = 0, \dots, d\times \omega - 1$. Based on \eqref{eq:pmf-s-fft}, we propose Algorithm \ref{algo:pmf-s} to compute the pmf for $S$ when margins are discrete. 

\begin{algorithm}
	\KwIn{Values of $p_{k, j}, j = 0, \dots, \omega-1, k = 1, \dots, d$, table $f_{\boldsymbol{I}}$}
	\KwOut{pmf of $S$}
	\nl \For{$k = 1, \dots, d$}{
		\nl Set $\boldsymbol{p}_k = (p_{k, 0}, \dots, p_{k, \omega - 1}, 0, \dots, 0) \in [0, 1]^{d\times \omega}$\;
		\nl Compute $\boldsymbol{P}_k$ as the cumulative sum of $\boldsymbol{p}_k$\;
		\nl Compute $\boldsymbol{P}_{k, [2]} = \boldsymbol{P}_{k}^2$ (element-wise)\;
		\nl Compute $\boldsymbol{p}_{k, [2]}$ as the difference vector of $\boldsymbol{P}_{k, [2]}$\;
		\nl Compute $\boldsymbol{p}_{k, [1]} = 2 \times \boldsymbol{p}_k - \boldsymbol{p}_{k, [2]}$ (element-wise) \;
		\nl Use fft to compute the discrete Fourier transform $\boldsymbol{\phi}_{k, [1]}$ of $\boldsymbol{p}_{k, [1]}$\;
		\nl Use fft to compute the discrete Fourier transform $\boldsymbol{\phi}_{k, [2]}$ of $\boldsymbol{p}_{k, [2]}$\;
	}
	\nl Compute $\boldsymbol{\phi}_{S} = \sum_{\boldsymbol{i} \in \{0, 1\}^d} f_{\boldsymbol{I}}(\boldsymbol{i})\prod_{k = 1}^{d}\boldsymbol{\phi}_{k, [i_k + 1]}$ (element-wise)\;
	\nl Use fft to compute the inverse discrete Fourier transform $\boldsymbol{p}_{S}$ of $\boldsymbol{\phi}_{S}$\;
	\nl Return $\boldsymbol{p}_{S}$.
	\caption{Computing the pmf of $S$.} \label{algo:pmf-s}
\end{algorithm}

\section{Approximation methods}\label{sec:approx}

As stated in Section \ref{sec:continuous}, the cdf or moments of the aggregate rv has a convenient form when the margins of the random vector are closed under order statistics. However, this is not the case for most continuous distributions. Fortunately, one may discretize continuous rvs into discrete rvs, and use the numerical tools provided in Section \ref{sec:discrete} to study the approximate behaviour of $S$. Some approximation methods are provided in \cite{embrechts2009panjer} or appendix E of \cite{klugman2018loss}. 

Since discretization is an approximation method, one wishes to obtain upper and lower bounds for the true cdf and risk measures of the aggregate rv in order to quantify the accuracy of the approximation. To construct these bounds we will require the following univariate stochastic order. 
\begin{definition}[Usual stochastic order]
	Let $Y$ and $Y'$ be two rvs with $\overline{F}_{Y}(x) \geq \overline{F}_{Y'}(x)$ for all $x \in \mathbb{R}$. Then, we say that $Y$ is smaller than $Y'$ under the usual stochastic order, and denote this relation by $Y \preceq_{st} Y'$. 
\end{definition}
Implications of the usual stochastic order between two rvs $Y$ and $Y'$ are that $E[Y] \leq E[Y']$, $\text{VaR}_{\kappa}(Y) \leq \text{VaR}_{\kappa}(Y')$ for all $\kappa \in (0, 1)$ and $E[\phi(Y)] \leq E[\phi(Y')]$ for all increasing function $\phi$, assuming that the expectations exist (including the TVaR). Within the context of the current paper, if we may construct cdfs for two rvs $A$ and $B$ such that $A \preceq_{st} S \preceq_{st} B$, then we may construct bounds on $VaR$s and certain expected values. One approach is to use the upper and lower methods, defined next. 

\begin{definition}[Lower method]
	The probability mass function of a discretized rv $\tilde{X}^{(l, h)}$, under the \emph{lower} method is, for $h > 0$, 
	$$\begin{cases}
		f_{\tilde{X}^{(l, h)}}(0) = 0&\\
		f_{\tilde{X}^{(l, h)}}(jh) = F_{X}(jh) - F_{X}((j-1)h), & \text{ for } j \in \mathbb{N}_1 .
	\end{cases}$$
\end{definition}

\begin{definition}[Upper method]
	The probability mass function of a discretized rv $\tilde{X}^{(u, h)}$ under the \emph{upper} method is, for $h > 0$, 
	$$\begin{cases}
		f_{\tilde{X}^{(u, h)}}(0) = F_{X}(h)&\\
		f_{\tilde{X}^{(u, h)}}(jh) = F_{X}((j+1)h) - F_{X}(jh), & \text{ for } j \in \mathbb{N}_1.
	\end{cases}$$
\end{definition}
Note that the authors of \cite{embrechts2009panjer} call the lower and upper methods, respectively, the backward and forward differences. From the definitions of the lower and upper methods, we have (see Section 1.11 of \cite{muller2002comparison}) that, for $0 < h < h' < \infty$,
\begin{equation}\label{eq:compare-discrete}
	\tilde{X}^{(u, h')} \preceq_{st} \tilde{X}^{(u, h)} \preceq_{st} X \preceq_{st} \tilde{X}^{(l, h)} \preceq_{st} \tilde{X}^{(l, h')}.
\end{equation}
It is useful to construct bounds as in \eqref{eq:compare-discrete} for the aggregate rv $S$. To do so, we first discretize the cdfs of each marginal distribution, in particular, the cdfs of the order statistics of each marginal. 

\begin{remark}\label{prop:lower-upper-order-stat}
	Let $X_1$ and $X_2$ be independent copies of a positive rv $X$ with cdf $F_X$. Let $X_{[1]} = \min(X_1, X_2)$ and $X_{[2]} = \max(X_1, X_2)$. Then,
	$$F_{\tilde{X}_{[1]}^{(m, h)}}(x) = 1 - (1 - F_{\tilde{X}^{(m, h)}}(x))^2; \quad 
	F_{\tilde{X}_{[2]}^{(m, h)}}(x) = F_{\tilde{X}^{(m, h)}}(x)^2,$$
	for $m \in \{l, u\}$ and $x \geq 0$. That is, one can compute the cdf of $\tilde{X}_{[1]}$ and $\tilde{X}_{[2]}$ using the definitions of the lower and upper discretization methods or first discretize the rv $X$ and then compute the cdf using the relationships in \eqref{eq:cdf-pdf-min} and \eqref{eq:cdf-pdf-max}. 
\end{remark}

We need the following proposition.
\begin{proposition}\label{prop:order-x-multivariate-usual}
	Let $\boldsymbol{X}$ be a random vector with cdf $F_{\boldsymbol{X}}(\boldsymbol{x}) = C(F_{X_1}(x_1), \dots, F_{X_d}(x_d)),$
	for $\boldsymbol{x} \in \mathbb{R}^d_+$ and $C \in \mathcal{C}_d^{FGM}$. Define the discretized random vectors $\tilde{\boldsymbol{X}}^{(m, h)} := (\tilde{X}_1^{(m, h)}, \dots, \tilde{X}_d^{(m, h)})$ for $m \in \{l, u\}$. Then, for $0 < h < h' < \infty$, we have
	$$\tilde{\boldsymbol{X}}^{(u, h')} \preceq_{st} \tilde{\boldsymbol{X}}^{(u, h)} \preceq_{st} \boldsymbol{X} \preceq_{st} \tilde{\boldsymbol{X}}^{(l, h)} \preceq_{st} \tilde{\boldsymbol{X}}^{(l, h')}.$$
\end{proposition}
\begin{proof}
    From \eqref{eq:compare-discrete}, we have that
    $\tilde{X}_k^{(u, h')} \preceq_{st} \tilde{X}_k^{(u, h)} \preceq_{st} X_k \preceq_{st} \tilde{X}_k^{(l, h)} \preceq_{st} \tilde{X}_k^{(l, h')}$, for $k = 1, \dots, d$. Further, the random vectors $\tilde{\boldsymbol{X}}^{(u, h')}$, $\tilde{\boldsymbol{X}}^{(u, h)}$, $\boldsymbol{X}$, $\tilde{\boldsymbol{X}}^{(l, h)}$ and $\tilde{\boldsymbol{X}}^{(l, h')}$ share the same copula, hence we obtain the desired result by applying Theorem 4.1 of \cite{muller2001stochastic} (see also Theorem 3.3.8 of \cite{muller2002comparison}).
\end{proof}
Define the aggregate rv of the discretized marginals with the upper and lower methods, that is, $\tilde{S}^{(u, h)}$ and $\tilde{S}^{(l, h)}$ by
$\tilde{S}^{(u, h)} = \tilde{X}^{(u, h)}_1 + \dots + \tilde{X}^{(u, h)}_d$ and $\tilde{S}^{(l, h)} = \tilde{X}^{(l, h)}_1 + \dots + \tilde{X}^{(l, h)}_d$. Since the usual stochastic order is preserved under monotone transformations (see Theorem 3.3.11 of \cite{muller2002comparison}), it follows from Proposition \ref{prop:order-x-multivariate-usual} that 
\begin{equation}\label{eq:order-s}
	\tilde{S}^{(u, h')} \preceq_{st} \tilde{S}^{(u, h)} \preceq_{st} S \preceq_{st}\tilde{S}^{(l, h)} \preceq_{st}\tilde{S}^{(l, h')},
\end{equation}
for $0 < h \leq h'$. The relationship in \eqref{eq:order-s} is useful since one can construct bounds on risk measures, in particular, construct upper and lower bounds on the stop-loss premiums, on the Value-at-Risk and on the Tail Value-at-Risk. For the latter, setting $0 < \kappa < 1,$ we have
\begin{equation}\label{eq:tvar-s}
	\text{TVaR}_{\kappa}\left(\tilde{S}^{(u, h')}\right) \leq \text{TVaR}_{\kappa}\left(\tilde{S}^{(u, h)}\right) \leq \text{TVaR}_{\kappa}\left(S\right) \leq \text{TVaR}_{\kappa}\left(\tilde{S}^{(l, h)}\right) \leq \text{TVaR}_{\kappa}\left(\tilde{S}^{(l, h')}\right).
\end{equation}

In the following example, we consider a portfolio of log-normal risks. Note that there are no closed-form expressions for the cdf of the minimum of two log-normal distributions, but one may still approximate the cdf with Algorithm \ref{algo:pmf-s}, and the relation in \eqref{eq:order-s} provides bounds on the tail value at risk of the aggregate distribution.

\begin{example}
	Consider a portfolio of $n = 3$ risks and $X_k \sim LNorm(\mu_k, \sigma_k)$ for $k = 1, 2, 3$. We set $(\mu_k, \sigma_k)$, $k \in \{1, 2, 3\}$ such that $E[X_k] = 10$ for $k = 1, 2, 3$ and $Var(X_1) = 20$, $Var(X_2) = 50$ and $Var(X_3) = 100$. The dependence structure is induced by a Markov-Bernoulli FGM copula, as introduced in \cite{blier-wong2022stochastic}, whose expression is
	\begin{equation*} 
		C\left( \boldsymbol{u}\right) = \prod_{k = 1}^{d}u_k\left(1 + \sum_{k = 1}^{\left\lfloor \frac{d}{2}\right\rfloor} \sum_{1\leq j_1 < \dots < j_{2k} \leq d}\alpha ^{\gamma_{j_1 \dots j_{2k}}}\overline{u}_{j_1}\dots \overline{u}_{j_{2k}}\right), \quad \boldsymbol{u} \in [0, 1]^d,
	\end{equation*}
	where $\gamma_{j_1 \dots j_{2k}} = \sum_{l = 1}^{k} \left(j_{2l} - j_{2l-1}\right)$ and dependence parameter satisfies $\alpha \in [-1, 1]$. For this example, we select the dependence parameter $\alpha = 0.5$. We aim to approximate the cdf of $S = X_1 + X_2 + X_3$ through discretization methods and using Algorithm \ref{algo:pmf-s}. 
	
	Figure \ref{fig:disc-lnorm} presents the cdf of $\tilde{S}^{(m, h)}$ for $m \in \{l, u\}$ and $h \in \{0.5, 1, 2\}$. Clearly, the relationship in \eqref{eq:order-s} is satisfied, it follows that the cdf of the continuous aggregate rv is between the green (lower method) and blue (upper method) curves. Table \ref{tab:tvar-discretization} presents the values of the TVaR risk measure at levels $\kappa \in \{0.9, 0.99, 0.999\}$ for the rvs $\tilde{S}^{(m, h)}$ for $m \in \{l, u\}$ and $h \in \{0.1, 0.5, 1, 2\}$. The relationship in \eqref{eq:tvar-s} is also satisfied. One can state, therefore, that $60.32 \leq \text{TVaR}_{0.9}(S) \leq 60.73$, without ever knowing the true cdf of $S$. Also, one may decrease the range between the lower and upper bounds by selecting a smaller span $h$ at the cost of more computations. For instance, selecting $h = 0.01$ yields an interval 
	$60.56 \leq \text{TVaR}_{0.9}(S) \leq 60.59$, but the computation time goes from 0.01 seconds for $h = 0.1$ to 64 seconds for $h = 0.01$.
		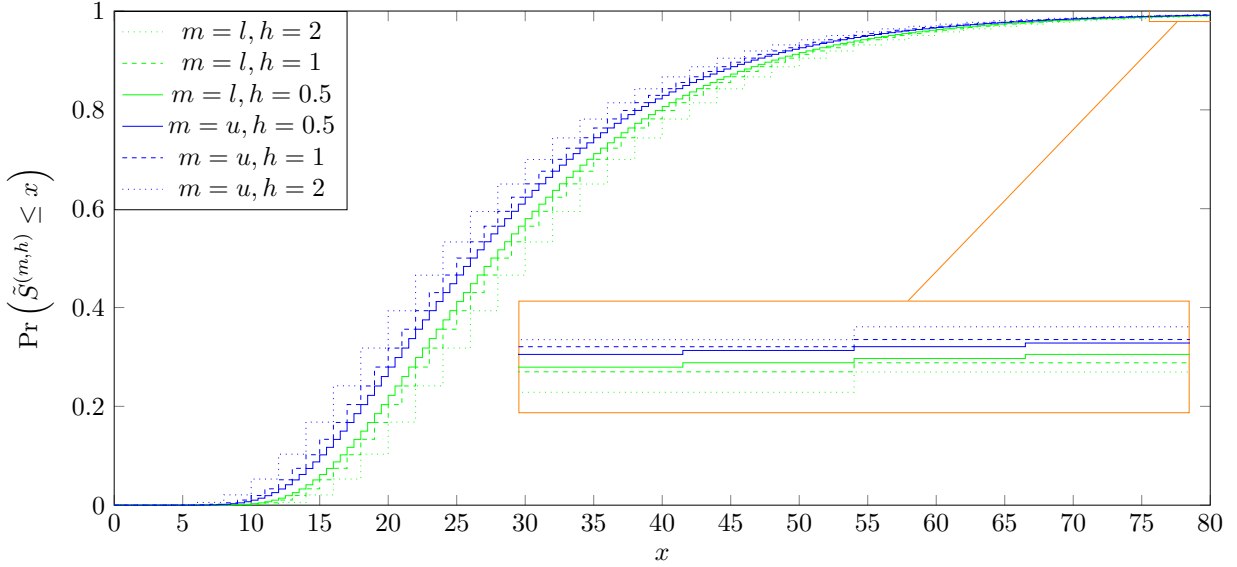
\begin{figure}[ht]
		\centering
		\resizebox{\textwidth}{!}{
		\begin{tikzpicture}
			\begin{axis}[
				width = 7in, 
				height = 3.5in,
				ymin = 0,
				xmin = 0, 
				ymax = 1, 
				xmax = 80,
				xlabel={$x$},
				ylabel={$\Pr\left(\tilde{S}^{(m, h)} \leq x\right)$}, 
				legend style={at={(0,1)},anchor=north west},
					]
				\addplot[const plot, dotted, green] table [x="x", y="FsL", col sep=comma] {code/Fs_upper_lower_3.csv};
				\addlegendentry{$m = l, h = 2$}
				\addplot[const plot, dash pattern=on 2pt off 2pt, green] table [x="x", y="FsL", col sep=comma] {code/Fs_upper_lower_2.csv};
				\addlegendentry{$m = l, h = 1$}
				\addplot[const plot, green] table [x="x", y="FsL", col sep=comma] {code/Fs_upper_lower_1.csv};
				\addlegendentry{$m = l, h = 0.5$}
				
				\addplot[const plot, blue] table [x="x", y="FsU", col sep=comma] {code/Fs_upper_lower_1.csv};
				\addlegendentry{$m = u, h = 0.5$}
				\addplot[const plot, dash pattern=on 2pt off 2pt, blue] table [x="x", y="FsU", col sep=comma] {code/Fs_upper_lower_2.csv};
				\addlegendentry{$m = u, h = 1$}
				\addplot[const plot, dotted, blue] table [x="x", y="FsU", col sep=comma] {code/Fs_upper_lower_3.csv};
				\addlegendentry{$m = u, h = 2$}
				
				\newcommand*\spyfactor{5}
				\newcommand*\spypoint{axis cs:78, 0.99}
				\newcommand*\spyviewer{axis cs:54, 0.3}
				\node[rectangle, draw, minimum height=0.01\textwidth, minimum width=0.06\textwidth, inner sep=0pt, orange] (spypoint) at (\spypoint) {};
				\node[rectangle, draw, minimum height=0.1\textwidth, minimum width=0.6\textwidth, inner sep=0pt, orange] (spyviewer) at (\spyviewer) {};
				\draw[orange] (spypoint) edge (spyviewer);
				\begin{scope}
					\clip (spyviewer.south west) rectangle (spyviewer.north east);
					\begin{scope}[scale around={\spyfactor:($(\spyviewer)!\spyfactor^2/(\spyfactor^2-1)!(\spypoint)$)}, yscale=1]
						\addplot[const plot, dotted, green] table [x="x", y="FsL", col sep=comma] {code/Fs_upper_lower_3.csv};
						\addplot[const plot, dash pattern=on 2pt off 2pt, green] table [x="x", y="FsL", col sep=comma] {code/Fs_upper_lower_2.csv};
						\addplot[const plot, green] table [x="x", y="FsL", col sep=comma] {code/Fs_upper_lower_1.csv};
						\addplot[const plot, blue] table [x="x", y="FsU", col sep=comma] {code/Fs_upper_lower_1.csv};
						\addplot[const plot, dash pattern=on 2pt off 2pt, blue] table [x="x", y="FsU", col sep=comma] {code/Fs_upper_lower_2.csv};
						\addplot[const plot, dotted, blue] table [x="x", y="FsU", col sep=comma] {code/Fs_upper_lower_3.csv};
					\end{scope}
				\end{scope}		
			\end{axis}
		\end{tikzpicture}}
		\caption{Cumulative distribution function for different discretization methods.}\label{fig:disc-lnorm}
	\end{figure}
	\begin{table}[ht]
		\centering
		\caption{Values of TVaR from different discretization methods.}\label{tab:tvar-discretization}
		\begin{tabular}{lrrrrrrrr}
			& \multicolumn{4}{c}{Upper} & \multicolumn{4}{c}{Lower} \\
			\cmidrule(r{4pt}){2-5} \cmidrule(l){6-9}
			$h$              &      2 &      1 &    0.5 &    0.1 &    0.1 &    0.5 &      1 &      2 \\ \hline
			$\kappa = 0.9$   &  57.60 &  59.08 &  59.83 &  60.43 &  60.73 &  61.33 &  62.08 &  63.60 \\
			$\kappa = 0.99$  &  92.65 &  94.13 &  94.88 &  95.48 &  95.78 &  96.38 &  97.13 &  98.65 \\
			$\kappa = 0.999$ & 142.93 & 144.42 & 145.16 & 145.76 & 146.06 & 146.66 & 147.42 & 148.93 \\ \hline
		\end{tabular}
	\end{table}
\end{example}

\section{Risk allocation and risk sharing}\label{sec:allocation}

It is natural, in the context of risk management, to study the impact of aggregating risks in an insurance portfolio or pool. To do so, we will study allocation rules, a problem related to aggregation of rvs. Allocations have actuarial applications in peer-to-peer insurance and in regulatory capital allocation. Throughout this section, we consider once again a portfolio of $d$ risks, each of which follow mixed Erlang distributions with common rate parameter. The dependence structure is once again induced by a FGM copula. From Section \ref{ss:mx-erlang-sum}, we know that the aggregate rv is also mixed Erlang distributed. 

The expected allocation, given by the expression 
$$E\left[X_m \times 1_{\{S = s\}}\right] = \int_{0}^{s} x f_{X_m, S_{-m}}(x, s-x) \diff x,$$
for $s \geq 0$, where $S_{-m} = \sum_{k = 1, k \neq m}^{d}X_k$, is a preliminary tool to develop allocation rules. The following proposition presents the expected allocation for mixed Erlang distributions with FGM dependence. 

\begin{proposition}\label{lemma:expected-allocation}
Let $\boldsymbol{X}$ be a random vector with cdf $F_{\boldsymbol{X}}(\boldsymbol{x}) = C(F_{X_1}(x_1), \dots, F_{X_d}(x_d))$, where $F_{X_k}$ is the cdf of a mixed Erlang rv, for $k = 1,\dots,d$. Further, assume that $F_{X_1}, \dots, F_{X_k}$ share the same rate parameter, and that $C \in \mathcal{C}_d^{FGM}$. For $m \in \{1, \dots, d\}$, the expected allocation is given by 
	\begin{align}
		E\left[X_m \times 1_{\{S = s\}}\right] &= \sum_{\boldsymbol{i} \in \{0, 1\}^d} f_{\boldsymbol{I}} (\boldsymbol{i}) \left[\sum_{l = 2}^{\infty} \sum_{\ell = 1}^{l-1} \Pr\left(L_{m, \{i_m + 1\}} = \ell\right)\times \vphantom{\Pr\left(\sum_{k = 1, k \neq m}^{d} L_{k, \{i+i_k\}} = l -  \ell\right)} \right. \nonumber\\
		& \left.\qquad \qquad \qquad \qquad \Pr\left(\sum_{k = 1, k \neq m}^{d} L_{k, \{i_k + 1\}} = l -  \ell\right) \frac{\ell}{2\beta} h(s; l + 1, 2\beta)\right],\label{eq:expected-allocation-mxerl}
	\end{align}
	where $h(s; \alpha, \beta)$ is the pdf associated with an Erlang distribution, for $s \geq 0$. Further, we have 
	\begin{align}
		E\left[X_m \times 1_{\{S > s\}}\right] &= \sum_{\boldsymbol{i} \in \{0, 1\}^d} f_{\boldsymbol{I}} (\boldsymbol{i}) \left[\sum_{l = 2}^{\infty} \sum_{\ell = 1}^{l-1} \Pr\left(L_{m, \{1+i_m\}} = \ell\right)\times \vphantom{\Pr\left(\sum_{k = 1, k \neq m}^{d} L_{k, \{i+i_k\}} = l -  \ell\right)} \right. \nonumber\\
		& \left.\qquad \qquad \qquad \qquad \Pr\left(\sum_{k = 1, k \neq m}^{d} L_{k, \{i+i_k\}} = l -  \ell\right) \frac{\ell}{2\beta} \overline{H}(s; l + 1, 2\beta)\right],\label{eq:expected-cumul-allocation-mxerl}
	\end{align}
	where $\overline{H}(s; \alpha, \beta) = 1 - H(s; \alpha, \beta)$, for $s \geq 0$.
\end{proposition}
\begin{proof}	
    From the joint LST in \eqref{eq:lst-x}, we condition on $\boldsymbol{I}$ to notice that the bivariate random vector $(X_m, S_{-m})$ is a mixture of independent bivariate random vectors. The bivariate LST of $(X_m, S_{-m})$ is
	\begin{align*}
		\mathcal{L}_{X_m, S_{-m}}(t_1, t_2) &= E_{\boldsymbol{I}}\left[\mathcal{P}_{L_{m, \{1+I_m\}}}\left(\frac{2\beta}{2\beta + t_1}\right) \prod\limits_{k=1, k \neq m}^{d} \mathcal{P}_{L_{k, \{1+I_k\}}}\left(\frac{2\beta}{2\beta + t_2}\right)\right]\\
		&= \sum_{\boldsymbol{i} \in \{0, 1\}^d} f_{\boldsymbol{I}} (\boldsymbol{i}) \mathcal{P}_{L_{m, \{1+i_m\}}}\left(\frac{2\beta}{2\beta + t_1}\right) \prod\limits_{k=1, k \neq m}^{d} \mathcal{P}_{L_{k, \{1+i_k\}}}\left(\frac{2\beta}{2\beta + t_2}\right),
	\end{align*}
	for $(t_1, t_2) \in \mathbb{R}_+^2$. Then, the expected allocation is
	\begin{align*}
		E\left[X_m \times 1_{\{S = s\}}\right] &= \int_{0}^{s} x f_{X_m, S_{-m}}(x, k-x) \diff x\\
		&= \int_{0}^{s} x \left[  \sum_{\boldsymbol{i} \in \{0, 1\}^d} f_{\boldsymbol{I}} (\boldsymbol{i}) f_{X_m, \{1 + i_m\}}(x) f_{\sum_{k = 1, k \neq m}^d X_{k, \{1 + i_k\}}}(s - x)  \right] \diff x\\
		&= \sum_{\boldsymbol{i} \in \{0, 1\}^d} f_{\boldsymbol{I}} (\boldsymbol{i}) \left[\int_{0}^{s} x f_{X_{m, \{1 + i_m\}}}(x) f_{\sum_{k = 1, k \neq m}^d X_{k, \{1 + i_k\}}}(s - x)\diff x\right].
	\end{align*}
    The result in \eqref{eq:expected-allocation-mxerl} follows from Propositions 4 and 5 of \cite{cossette2012tvarbased} since each integral is an expectation from a pair of independent mixed Erlang rvs. Finally, \eqref{eq:expected-cumul-allocation-mxerl} follows from integrating \eqref{eq:expected-allocation-mxerl} on the interval $(s, \infty)$. 
\end{proof}

We demonstrate the usefulness of Proposition \ref{lemma:expected-allocation} in the following subsections. 

\subsection{Conditional mean risk sharing}

The rise of peer-to-peer insurance has ignited a lot of interest in risk allocation and risk-sharing rules. A participant to a pool of insurance risk should pay a contribution relative to the risk he contributes to the pool, hence one seeks fair risk sharing rules to determine this value, see, for instance, \cite{denuit2019sizebiased}, \cite{denuit2020investing} or \cite{denuit2021risksharing} for discussions. The conditional mean risk sharing is one such rule, where the participant pays his expected contribution, given the total realized losses (denoted $s$, with $s \geq 0$) in the pool, that is, 
$$E\left[X_k \vert S = s\right] = \frac{E\left[X_k \times 1_{\{S = s\}}\right]}{f_{S}(s)},$$
see \cite{denuit2012convex}, \cite{denuit2021risk} for details and properties of the conditional mean risk sharing rule. When each risk is mixed Erlang distributed and the dependence structure is induced by a FGM copula, the numerator is given by Proposition \ref{lemma:expected-allocation} and, taking $q_{S, j}$ for $j \in \mathbb{N}_1$ as defined in \eqref{eq:cdf-mx-erl}, the denominator is
$$f_S(s) = \sum_{j = 1}^\infty q_{S, j} h(s; j, 2\beta),$$
for $s \geq 0$. It follows that the conditional mean risk sharing rule has a convenient closed-form expression.

\begin{example}
\label{ex:SixRisksMixedErlang}
    We consider a portfolio of six risks where each follows a mixed Erlang distribution with common rate parameter $1/2$ with cdfs
	$$F_{X_k}(s) = \Pr(X_k \leq s) = \sum_{j = 1}^{\infty} q_{k, j} H(s; j, 1/2),$$
	for $k \in \{1, \dots, 6\}$, and 
	$$\begin{cases}
		q_{1, j} = 1 \times 1_{\{j = 1\}}; \\
		q_{2, j} = \frac{1}{2}^{j} \times 1_{\{j \in \mathbb{N}_1\}}; \\
		q_{3, j} = 5^{j - 1} e^{-5}/(j-1)! \times 1_{\{j \in \mathbb{N}_1\}}; \\
		q_{4, j} = \Gamma(j-1+2)/\Gamma(2)/(j-1)! 0.25^2 (0.75)^{j-1} \times 1_{\{j \in \mathbb{N}_1\}};\\
		q_{5, j} = 10^{j - 1} e^{-10}/(j-1)! \times 1_{\{j \in \mathbb{N}_1\}}; \\
		q_{6, j} = \Gamma(j-1+3)/\Gamma(3)/(j-1)! 0.2^3 (0.8)^{j-1} \times 1_{\{j \in \mathbb{N}_1\}}.
	\end{cases}$$
	For convenience, we artificially construct vectors of probabilities whose masses correspond to known discrete distributions (Dirac, geometric, Poisson, negative binomial); this will help us control the shape of the marginal distributions. Notice that the risks are highly heterogeneous, since $\{q_{6, j}, j\in \mathbb{N}_1\}$ comes from a distribution with a heavier tail than $\{q_{1, j}, j\in \mathbb{N}_1\}$. Also note that $E[X_4] < E[X_5]$, but $Var(X_4) > Var(X_5)$. In Table \ref{tab:summary-ex-allocation}, we present the values of risk measures individually applied to each risk. 
	\begin{table}[ht]
		\centering
		\caption{Summary description for marginal rvs.}\label{tab:summary-ex-allocation}
		\begin{tabular}{lrrrrrr}
			$k$                       &         1 &        2 &        3 &         4 &        5 &         6 \\ \hline
			$E[X_k]$                  &  2 &  4 & 12 &  14 & 22 &  26 \\
			$Var(X_k)$                &  4 & 16 & 44 & 124 & 84 & 292 \\
			$\text{VaR}_{0.99}(X_k)$  &  9.21 & 18.42 & 31.44 &  50.86 & 47.45 &  79.72 \\
			$\text{TVaR}_{0.99}(X_k)$ & 11.21 & 22.42 & 35.40 &  59.90 & 52.30 &  92.03
		\end{tabular}
	\end{table}
	
	For every dependence structure, we have $E[S] = 80$. In Table \ref{tab:cmrs-mx-erl}, we present the realizations of random vectors under the conditional mean risk sharing rule, when the aggregate rv $S$ takes either the value of $E[S]/2$, $E[S]$, or $2\times E[S]$. 
	We observe a surprising pattern that, to the best of our knowledge, has not been previously remarked upon. The rv $X_1$ ($X_6$) is the safest (riskiest), having the smallest (largest) mean, variance, VaR and TVaR at level 0.99. For $s = 40$, we observe that \emph{increasing} the dependence (according to the supermodular order) results in a decrease (increase) of the conditional mean for the rv $X_1$ ($X_6$). For $s = 160$, we observe the opposite pattern: \emph{increasing} the dependence (according to the supermodular order) results in an increase (decrease) of the conditional mean for the rv $X_1$ ($X_6$). Further, observe that when $S = s$, the smallest conditional mean for $X_2$ and $X_3$ occurs when the dependence structure is independence. 
	\begin{table}[ht]
	    \centering
	    \caption{Outcomes for risk premiums under the conditional mean risk sharing rule.}
	    \label{tab:cmrs-mx-erl}
	    \begin{tabular}{crrrrrrr}
                                    $s$	& &    \multicolumn{6}{c}{$E[X_k | S = s]$}      \\
        	\cmidrule(l){3-8}
        	                       &        & $k = 1$  & $k = 2$  &  $k = 3$  &  $k = 4$  &  $k = 5$  &  $k = 6$  \\ \hline
        	\multirow{3}{*}{$40$}  &  END   & 1.928175 & 2.987516 & 7.996234  & 5.766606  & 13.401958 & 7.919511  \\
        	                       &  IND   & 1.575428 & 2.551020 & 7.668274  & 5.699930  & 13.761121 & 8.744228  \\
        	                       &  EPD   & 0.941819 & 1.757806 & 7.136790  & 5.658961  & 14.296102 & 10.208524 \\ \hline
        	\multirow{3}{*}{$80$}  &  END   & 2.030938 & 4.123420 & 12.407195 & 13.910778 & 22.776325 & 24.751343 \\
        	                       &  IND   & 2.042401 & 4.106984 & 12.392149 & 13.867892 & 22.741896 & 24.848677 \\
        	                       &  EPD   & 2.205948 & 4.149484 & 12.499946 & 13.398856 & 22.671998 & 25.073768 \\ \hline
        	\multirow{3}{*}{$160$} &  END   & 1.721004 & 4.234335 & 13.912178 & 30.207898 & 27.145704 & 82.778881 \\
        	                       &  IND   & 2.330977 & 5.554256 & 15.892004 & 31.485783 & 29.453031 & 75.283950 \\
        	                       &  EPD   & 3.347377 & 7.541924 & 18.660443 & 32.720014 & 32.458935 & 65.271307 \\ \hline
        \end{tabular}
	\end{table}
\end{example}

\subsection{Risk allocation based on Euler's rule}

For regulatory and capital requirement purposes, one must often decompose aggregate risk measures to the individual risks that contributed to it. The TVaR is a popular risk measure since it is coherent. The TVaR of a continuous rv is also called the conditional tail expectation, see, for instance, \cite{artzner1999application} \cite{artzner1999coherent}, \cite{acerbi2001expected}, \cite{acerbi2002coherence} for motivations and properties of the conditional tail expectation for risk management. When one establishes global capital with the TVaR, one may deconstruct this risk measure to TVaR-based allocations with the help of Euler's risk allocation rule (\cite{tasche1999risk}, \cite{denault2001coherent}). For $\kappa \in (0, 1)$, the TVaR-based allocation for continuous rvs is given by $\text{TVaR}_{\kappa}(X_j; S) = E[X_j \times 1_{\{S > \text{VaR}_{\kappa}(S)\}}]/(1-\kappa)$, for $j \in \{1, \dots, d\}$. Within the context of this paper, applying Proposition \ref{lemma:expected-allocation}, we have
\begin{align}
	\text{TVaR}_{\kappa}(X_j; S) &= \frac{1}{1 - \kappa} \left\{\sum_{\boldsymbol{i} \in \{0, 1\}^d} f_{\boldsymbol{I}} (\boldsymbol{i}) \left[\sum_{l = 2}^{\infty} \sum_{\ell = 1}^{l-1} \Pr\left(L_{m, \{i_m + 1\}} = \ell\right)  \times \right.\right.\nonumber\\
	& \qquad \qquad \left.\left. \Pr\left(\sum_{k = 1, k \neq m}^{d} L_{k, \{i_k + 1\}} = l -  \ell\right) \frac{\ell}{2\beta} \overline{H}(\text{VaR}_{\kappa}(S); l + 1, 2\beta)\right]\right\},\label{eq:tvar-contribution}
\end{align}
for $j \in \{1, \dots, n\}$, where we compute $\text{VaR}_{\kappa}(S)$ with numerical optimization of \eqref{eq:cdf-mx-erl}. One may verify that $$\sum_{k = 1}^{d}\text{TVaR}_{\kappa}(X_k; S) = \text{TVaR}_{\kappa}(S).$$

The result in \eqref{eq:tvar-contribution} was also developed in equation (35) of \cite{cossette2013multivariate}, but the formula is very tedious.
By studying multivariate mixed Erlang distributions from the order statistic perspective, and the FGM copula from the stochastic representation, one has an intuitive understanding of the underlying stochastic phenomenon and obtains straightforward expressions for the TVaR and TVaR-based allocation rules. Also, \eqref{eq:tvar-contribution} uses the stochastic formulation of the FGM copula (based on the symmetric multivariate Bernoulli random vector $\boldsymbol{I}$), which is more convenient in higher dimensions since most cases of interest (for instance, minimal and maximal dependence under the supermodular order for exchangeable FGM copulas) are easier to formulate with the stochastic representation. Also, the outer sum in \eqref{eq:tvar-contribution} is a sum over $2^d$ values, which could be computationally prohibitive, but for most special cases, including minimal and maximal dependence under the supermodular order for exchangeable FGM copulas, the pmf is non-zero for few vectors of $\boldsymbol{i} \in \{0, 1\}^d$.

\begin{example} 
Consider the portfolio of six risks as introduced in Example \ref{ex:SixRisksMixedErlang}. In Table \ref{tab:tvar-allocations}, we provide values of the TVaR-based risk allocation from the expression in \eqref{eq:tvar-contribution}. 
	\begin{table}[ht]
		\centering
		\caption{Values of $\text{TVaR}_{0.99}(X_k; S)$, for $k \in \{1, \dots, 6\}$ for different $C \in \mathcal{C}_d^{FGM}$.}
		\label{tab:tvar-allocations}
		\begin{tabular}{lrrrrrrrrr}
			                   & $Var(S)$ & $\text{VaR}_{0.99}(S)$ & $\text{TVaR}_{0.99}(S)$ &     \multicolumn{6}{c}{$\text{TVaR}_{0.99}(X_k; S)$}      \\
			\cmidrule(l){5-10} &          &                        &                         & $k = 1$ & $k = 2$ & $k = 3$ & $k = 4$ & $k = 5$ & $k = 6$ \\ \hline
			END                &   452.45 &                 140.58 &                  153.41 &    1.74 &    4.26 &   13.91 &   29.10 &   27.06 &   77.35 \\
			Ind                &      564 &                 146.71 &                  160.14 &    2.33 &    5.55 &   15.87 &   31.44 &   29.41 &   75.54 \\
			EPD                &  1121.77 &                 163.57 &                  177.24 &    3.39 &    7.79 &   19.08 &   36.48 &   33.23 &   77.25
		\end{tabular}
	\end{table}
	
    As shown in Section \ref{sec:stochastic-order}, we have $\boldsymbol{U}^{END} \preceq_{sm} \boldsymbol{U}^{Ind} \preceq_{sm} \boldsymbol{U}^{EPD}$, hence the corresponding aggregate rvs are ordered according to the convex order. This fact is verified from the size of the variance and the TVaR at level 0.99. Further, we observe for $k \in \{1, \dots, 5\}$ that $\mathrm{TVaR}_{0.99}(X_k; S)$ is smallest for the END FGM copula and largest for the EPD FGM copula. However, this is not the case for $X_6$, which is the riskiest in the portfolio. The authors were surprised to observe, for the rv $X_6$, that the smallest risk contribution occurs when the dependence structure is independence, while the largest risk contribution occurs with negative dependence. Investigating why this is the case represents an interesting avenue for future research. 
\end{example}

\section{Discussions}\label{sec:conclusion}

In this paper, we revisit risk aggregation and risk allocation with the FGM copula. By studying the problem using the stochastic representation of the FGM copula, we develop convenient representations for the cdf or moments of aggregate rvs when the dependence structure is induced by a FGM copula. One significant contribution of this work over the existing literature is our ability to order aggregate rvs according to stochastic orders. 

In Section \ref{sec:continuous}, we have provided convenient closed-form expressions for cdfs and moments of the aggregate rv $S$ for positive and continuous distributions. Other closed-from expressions are possible for continuous distributions. For instance, if $X$ has a cdf that is symmetric about $x = \mu$, we have $f_{X_{[1]}}(\mu + x) = f_{X_{[2]}}(\mu - x)$ and $\mu_{X_{[1]}}^{(m)} = (-1)^{m}\mu_{X_{[2]}}^{(m)}$. For $\mu = 0$, we have $X_{[1]} \overset{\mathcal{D}}{=} -X_{[2]}.$ It follows that 
$$E\left[S^m\right] = \sum_{j_1 + \cdots + j_d = m} \frac{m!}{j_1!\dots j_d!}E_{\boldsymbol{I}}\left[\prod_{k = 1}^{d}(-1)^{j_kI_k}\mu_{X_{[1]}}^{(j_k)}\right].$$
We leave the study of risk aggregation under FGM dependence of rvs whose support is on $\mathbb{R}$ as future research. 

In Section \ref{sec:allocation}, we presented numerical illustrations of conditional mean risk sharing and risk allocation based on Euler's rule for mixed Erlang marginals. Since the results of the current paper allow for exact expressions, and that the FGM copula admits multiple shapes of dependence (including negative dependence), we are in a position to investigate examples that provide apparent counter-intuitive results that were previously unknown (to the best of our knowledge) in the literature on risk sharing. Such results lead to open questions regarding the stochastic orderings of risk-sharing rules or ordering contributions based on Euler's rule or any other capital allocation rule.

\section{Acknowledgement}

This work was partially supported by the Natural Sciences and Engineering Research Council of Canada (Blier-Wong: 559169, Cossette: 04273; Marceau: 05605).

\end{document}